\documentclass{amsart}
\title{$K$-Knuth Equivalence for Increasing Tableaux}
\date{\today}
\author[Gaetz]{Christian Gaetz}
\author[Mastrianni]{Michelle Mastrianni}
\author[Patrias]{Rebecca Patrias}
\author[Peck]{Hailee Peck}
\author[Robichaux]{Colleen Robichaux}
\author[Schwein]{David Schwein}
\author[Tam]{Ka Yu Tam}

\usepackage{amsmath,amssymb, amsfonts}
\usepackage{todonotes}
\usepackage[vcentermath]{youngtab}
\usepackage{ytableau}
\ytableausetup{centertableaux}
\newcommand{\boxsize}{1.3em}
\usepackage{booktabs}
\usepackage{algcompatible}
\usepackage{algorithm}
\usepackage{algpseudocode}
\usepackage{verbatim}
\usepackage{mathtools}

\newtheorem{theorem}{Theorem}[section]
\newtheorem{lemma}[theorem]{Lemma}
\newtheorem{conjecture}[theorem]{Conjecture}
\newtheorem{corollary}[theorem]{Corollary}
\newtheorem{proposition}[theorem]{Proposition}

\newtheorem{problem}{Problem}
\newtheorem*{claim}{Claim}
\newtheorem{case}{Case}
\theoremstyle{definition}
\newtheorem{definition}[theorem]{Definition}
\newtheorem{example}[theorem]{Example}

\newcommand{\N}{\mathbb N}
\newcommand{\row}{\mathfrak{row}}
\newcommand{\col}{\mathfrak{col}}
\newcommand{\lis}{\mathfrak{lis}}
\newcommand{\lds}{\mathfrak{lds}}
\newcommand{\sss}{\scriptstyle} 
\newcommand{\pend}{\mathcal P_{\text{end}}}

\DeclareMathOperator{\swap}{swap}

\DeclareMathOperator{\arm}{arm}
\DeclareMathOperator{\leg}{leg}

\begin{document}

\begin{abstract}
A $K$-theoretic analogue of RSK insertion and the Knuth equivalence relations were introduced in \cite{BuKr06, BuSa13}. The resulting $K$-Knuth equivalence relations on words and increasing tableaux on~$[n]$ have
prompted investigation into the equivalence classes of tableaux
arising from these relations. Of particular interest are the tableaux
that are unique in their class, which we refer to as \textit{unique
rectification targets} (URTs). In this paper we give several new
families of URTs and a bound on the length of intermediate
words connecting two $K$-Knuth equivalent words. In addition, we
describe an algorithm to determine if two words are $K$-Knuth equivalent
and to compute all $K$-Knuth equivalence classes of tableaux on~$[n]$.
\end{abstract}

\maketitle

\section{Introduction}

In 2006, Buch et al.\ introduced a new combinatorial algorithm called \textit{Hecke insertion}, used to insert a word into an increasing tableau~\cite{BuKr06}. The algorithm is a $K$-theoretic analogue of the well-known Schensted algorithm for the insertion of a word into a semistandard Young tableau.

If two words insert into the same tableau via Schensted's insertion algorithm, they are said to be \textit{Knuth equivalent} and can be connected via the \textit{Knuth equivalence relations}. Knuth equivalence has a $K$-theoretic analogue referred to as $K$-Knuth equivalence, first defined in~\cite{BuSa13} and motivated by Thomas and Yong's $K$-theoretic jeu de taquin algorithm introduced in~\cite{ThYo07}. An important difference between Knuth equivalence and $K$-Knuth equivalence is that, while insertion equivalence via the Schensted algorithm (resp. the Hecke algorithm) implies Knuth equivalence (resp. $K$-Knuth equivalence), the converse holds for the standard version but not for the $K$-theoretic version. 
In other words, two words can be $K$-Knuth equivalent but insert into different tableaux via the Hecke insertion algorithm.

A $K$-Knuth equivalence class typically contains words from different insertion classes. There are some $K$-Knuth classes, however, for which all words in the class insert into the same tableau. A class with this property is called a \textit{unique rectification class}, and its corresponding insertion tableau is called a \textit{unique rectification target} (URT). In both~\cite{BuKr06} and~\cite{PaPy14}, Hecke insertion and $K$-Knuth equivalence were used to rederive a $K$-theoretic version of the Littlewood-Richardson rule for the cohomology rings of Grassmanians. 
In order to get a working version of this rule, non-URTs needed to be avoided. 
Hence Patrias and Pylyavskyy~\cite{PaPy14} posed the following natural question, an open problem.

\begin{problem} \label{problem}
Characterize all URTs or at least provide an efficient algorithm to determine if a given tableau is a URT.
\end{problem}

This paper makes partial progress toward
answering Problem~\ref{problem}. 
In more detail, we will extend previous results,
many from a paper ~\cite{BuSa13} of Buch and Samuel about URTs and $K$-Knuth equivalence.
In Section 2, we provide more background on Hecke insertion and $K$-Knuth equivalence and discuss the $K$-theoretic extension of the jeu de taquin algorithm of Thomas and Yong, which provides another way of determining whether two tableaux are in the same equivalence class. 
We also summarize Buch and Samuel's results about URTs and give invariants for classes of $K$-Knuth equivalent tableaux.

In Section 3, we give a finite-time algorithm to compute all $K$-Knuth classes of tableaux on a given alphabet $[n]$. From this we derive a finite-time algorithm to determine if two given words are $K$-Knuth equivalent. 

In Section 4, we show that every two $K$-Knuth equivalent tableaux on~$[n]$ can be connected by intermediate words of length at most $\frac{1}{3}n(n+1)(n+2) + 3$.
The proof of this bound also includes a useful lemma stating that any words of length $\ell$ in an insertion class can be connected to the row word of the corresponding insertion tableau by moving through intermediate words of length at most $\ell$.

Sections 5 and 6 respectively detail two new families of URTs: \textit{right-alignable tableaux} and \textit{hook-shaped tableaux}. We introduce the notion of a \textit{repetitive reading word} and use it to prove that all right-alignable tableaux are URTs. We also give a method for easily determining whether any given hook-shaped tableau is a URT based on the values of its entries.

Finally, in Section 7, we discuss various findings on the number of $K$-Knuth equivalence classes of tableaux on an alphabet $[n]$ and the number of unique rectification classes among them. We then give additional conjectures and related results.

\section{Background}
The goal of this section is to familiarize the reader
with the language of $K$-Knuth equivalence relations
on increasing tableaux, which for the most part
parallels the better-known 
Knuth equivalence relations~\cite{LaLeTh02}.

\subsection{Increasing Tableaux}\label{section:increasing tableaux}
In this section, we will define in more detail 
\emph{increasing tableaux} \cite{ThYo07},
the main subject of this paper,
as well as related terminology,
following the formalization of~\cite[Section 3.1]{BuSa13}.
Throughout this paper, $\mathbb N$ will denote
the set of positive integers.

Elements of the set~$\N\times\N$ are called \emph{boxes}
and will form the building blocks of increasing tableaux.
We will visualize $\N\times\N$ as an infinite matrix
comprised of boxes: the box~$(i,j)$
appears in row~$i$ and column~$j$.

Suppose $\alpha=(i_1,j_1)$ and $\beta=(i_2,j_2)$
are boxes.
We say that $\alpha$ is \textit{strictly northeast}
of $\beta$ if $i_1< i_2$ and $j_1>j_2$, 
and we say that $\alpha$ is \textit{weakly northeast} 
of $\beta$ if $i_1\leq i_2$ and $j_1\geq j_2$.
The reader can formulate the analogous definitions
for the remaining cardinal directions,
which we omit.
In addition, we say $\alpha$
is \textit{above} $\beta$ to mean $\alpha$
is north of $\beta$, we say $\alpha$ is \emph{directly above}
$\beta$ to mean $i_1=i_2+1$ and $j_1=j_2$,
and so on.

A \textit{shape} $\lambda$ is any finite subset of $\N\times\N$.
We say $\lambda$~is a \emph{straight shape} if
whenever $\lambda$ contains the box~$\alpha$
it contains all boxes weakly northwest of~$\alpha$.
A \textit{skew shape}~$\nu/\mu$
is the set-theoretic difference of two straight shapes
$\nu\supseteq\mu$. 

\begin{example} \label{example:shape}
Of the shapes below, the first is neither straight nor skew,
the second is skew but not straight, and the third is straight.
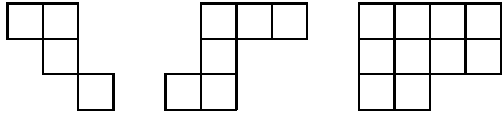
\[
\ytableausetup{boxsize=\boxsize}
\begin{ytableau}
{} & & \none \\
\none & & \none \\
\none & \none &
\end{ytableau}\qquad
\begin{ytableau}
\none & & & \\
\none & & \none \\
& & \none
\end{ytableau}\qquad
\begin{ytableau}
{} & & & \\
& & &\\
&
\end{ytableau}
\]
\end{example}

We can identify a straight shape with a partition
as follows.
Given a straight shape~$\lambda$,
let $\lambda_i$ denote the number
of boxes in row $i$.
If $\lambda$~has $\ell$ nonempty rows
then $\lambda$~is uniquely determined
by the tuple $(\lambda_1, \lambda_2, \cdots, \lambda_\ell)$.
By definition, $\lambda_1\geq \lambda_2\geq\cdots\geq\lambda_\ell$.
The straight shape given in Example~\ref{example:shape},
for instance, corresponds to the partition~$(4,4,2)$.

A \textit{filling} of a shape $\lambda$ is any map $T:\lambda\to\N$
that assigns an integer to each box of $\lambda$.
The image of a box $\alpha$ under $T$
is called the \textit{label} or \textit{entry} for $\alpha$. We say that the filling $T$ is an \textit{increasing tableau} (\textit{of shape $\lambda$}) if the entries of $T$ strictly increase down columns and
from left to right along rows,
that is, if $T(\alpha)<T(\beta)$ whenever
$\alpha$ is weakly northwest of and different from $\beta$. In this paper, all tableaux are increasing tableaux,
and in particular we will not consider semistandard tableaux.
A tableau $T$ of shape $\lambda$
is \emph{straight} if $\lambda$ is straight
and \emph{skew} if $\lambda$ is skew.
Unless otherwise mentioned, we will write
``tableau'' to mean ``straight tableau.''

\begin{example} \label{example:tableaux}
Of the fillings below, only the third is an increasing tableau.
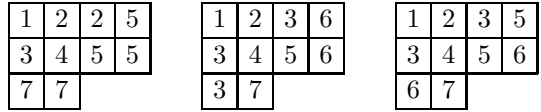
\[
\ytableausetup{boxsize=\boxsize}
\begin{ytableau}
1 & 2 & 2 & 5 \\
3 & 4 & 5 & 5\\
7 & 7
\end{ytableau}\qquad
\begin{ytableau}
1 & 2 & 3 & 6 \\
3 & 4 & 5 & 6\\
3 & 7
\end{ytableau}\qquad
\begin{ytableau}
1 & 2 & 3 & 5 \\
3 & 4 & 5 & 6\\
6 & 7
\end{ytableau}
\]
\end{example}

As with matrices, let $\lambda^t$
denote the \emph{transpose} of $\lambda$,
defined by
\[
\lambda^t = \{(j,i)\,:\,(i,j)\in\lambda\}.
\]
Let $T^t : \lambda^t\to\N$
denote the transpose of $T$,
defined by $T^t(j,i) = T(i,j)$.
The transpose of a tableau or shape 
is sometimes referred to as its \emph{conjugate}.

\begin{example}
The tableau
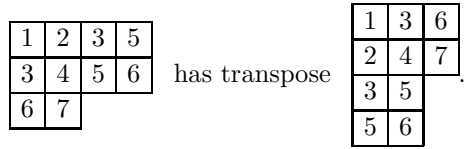
\[
\begin{ytableau}
1 & 2 & 3 & 5 \\
3 & 4 & 5 & 6\\
6 & 7
\end{ytableau}
\hspace{0.3cm}
\mbox{has transpose}
\hspace{0.3cm}
\begin{ytableau}
1 & 3 & 6 \\
2 & 4 & 7 \\
3 & 5 \\
5 & 6
\end{ytableau}.
\]
\end{example}

\begin{definition}
A tableau~$T$ of any shape is \emph{initial} if the set
of labels of~$T$ is~$[n]:=\{1,2,\ldots,n\}$ for some $n\in\N$.
A word~$w$ is \emph{initial} if the set of
letters appearing in~$w$ is~$[n]$
for some $n\in\N$.
\end{definition}
\begin{example}
The word~$124335$ is initial
but the word~$14355$ is not.
Of the two tableaux below, the left tableau
is initial and the right tableau is not.
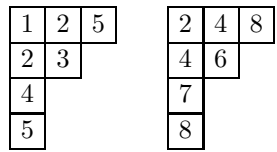
\[
\begin{ytableau}
1 & 2 & 5\\
2 & 3 \\
4\\
5
\end{ytableau}\qquad
\begin{ytableau}
2 & 4 & 8\\
4 & 6 \\
7\\
8
\end{ytableau}
\]
\end{example}
Initial tableaux are often
easier to work with, and for this reason
we will usually restrict our attention to initial tableaux.
This restriction comes at no loss of generality because we can relabel a tableau
without changing anything essential
provided the order relations between the labels are preserved. 
The following definition formalizes this notion.
\begin{definition}
Let $w=w_1w_2\dots w_k$ be a word and let
$a_1 < a_2 < \cdots < a_\ell$ be the ordered list of letters
appearing in~$w$. The \emph{standardization}
of~$w$ 
 is the word formed by replacing $a_i$ with $i$ in $w$.

Similarly, let $T$ be a tableau and let
$a_1 < a_2 < \cdots < a_\ell$ be the ordered list
of letters appearing in~$T$.
The standardization of~$T$ is the tableau
formed from $T$ by replacing every entry $a_i$ with $i$.
\end{definition}
\begin{example}
The standardization of the word
$35822$ is $23411$. The standardization
of the tableau
$\young(2568,59)$ is $\young(1234,25)$.
\end{example}

\subsection{Hecke Insertion}
Hecke insertion is an algorithm for inserting a positive integer into an increasing tableau, resulting in another increasing tableau,
which may or may not be the same as the original. Hecke insertion is a $K$-theoretic analogue of the standard Robinson-Schensted-Knuth (RSK) algorithm for the insertion of words into semistandard tableaux.
The elementary step of the Hecke insertion algorithm
is the insertion of a positive integer into a row of the tableau.
After the row is modified, either a new positive integer 
is inserted into the next row or the algorithm terminates.

The rules for Hecke inserting a positive integer~$x$ 
into row~$R$ of a tableau~$T$ are as follows.
Suppose first that $x\geq y$ for all $y\in R$.
\begin{enumerate}
\item If adjoining a box containing $x$ to the end of $R$ results in a 
valid increasing tableau~$T'$, 
then $T'$ is the result of the insertion,
and the algorithm terminates.
\item If adjoining a box containing $x$ to the end of $R$ does \emph{not} result in 
a valid increasing tableau, then $R$ is unchanged
and the algorithm terminates.
\end{enumerate}
Otherwise, let $y$ be the smallest integer in~$R$ that is strictly larger than~$x$.
\begin{enumerate}
\setcounter{enumi}{2} 
\item If replacing $y$ with~$x$ 
results in an increasing tableau, 
then replace $y$ with~$x$
and insert $y$~into the next row.
\item If replacing $y$ with~$x$ does \emph{not} result in an increasing tableau, then insert~$y$ into the next row and do not change~$R$.
\end{enumerate}
We write $T \leftarrow x$ to denote the final tableau
resulting from the Hecke insertion of~$x$ into the first row of~$T$.

It will occasionally be convenient to consider the \textit{column insertion} of $x$ into $T$,
which is computed by performing Hecke insertion with columns
playing the role of rows. 
Formally, the column insertion of $x$ into $T$ is given by $(T^t \leftarrow x)^t$.  From now on, ``insertion'' will always refer to \textit{Hecke insertion}.  
\begin{example}\cite[Example 2.3]{PaPy14}
$$\young(1235,2346,6,7) \leftarrow 3 = \young(1235,2346,6,7)$$
In this example, inserting 3 into the first row invokes rule (4), so we insert 5 into the second row.  
This invokes (4) again, so we insert 6 into the third row.  By (2), we get the tableau shown.
\end{example}
\begin{example}\cite[Example 2.4]{PaPy14}
$$\young(246,368,7) \leftarrow 5 = \young(245,368,78)$$
We first insert 5 into the first row, which by (3) replaces the 6 in the rightmost box, bumping the 6 into the second row.  By rule (4), the second row is unchanged, and we insert an 8 into the third row.  Rule (1) gives the resulting tableau.
\end{example}

Let $w=w_1\cdots w_n$ be a word.  
The \textit{insertion tableau} of $w$, written $P(w)$, 
is formed by recursively Hecke inserting
the letters of $w$ from left to right:
\[
P(w)=(\cdots((\varnothing \leftarrow w_1) \leftarrow w_2)\cdots\leftarrow w_n).
\]

\subsection{$K$-Knuth equivalence}

Just as Hecke insertion is a $K$-theoretic analogue of the standard RSK insertion, \textit{$K$-Knuth equivalence} is the corresponding analogue for Knuth equivalence. Recall that the Knuth equivalence relations on words in $\mathbb{N}$ with distinct letters are as follows:
\begin{align*}
xzy &\sim zxy,\qquad (x<y<z) \\
yxz &\sim yzx,\qquad (x<y<z).
\end{align*}
Two words are said to be \textit{Knuth equivalent} if one can be obtained from the other via a finite series of applications of the above Knuth relations.

In the $K$-theoretic case, we allow words to have repeated letters. The first two rules are precisely the same. However, we now have two additional rules with important consequences. The $K$-Knuth relations are as follows:
\begin{align*}
xzy &\equiv zxy,\qquad (x < y < z) \\
yxz &\equiv yzx,\qquad (x < y < z) \\
x &\equiv xx, \\
xyx &\equiv yxy.
\end{align*}
Again, two words are said to be \textit{$K$-Knuth equivalent} if one can be obtained from the other via a finite series of applications of the above $K$-Knuth relations. We will often refer to an application of a $K$-Knuth relation as a \textit{$K$-Knuth move}. In this terminology, two words are equivalent if one can be obtained from the other using finitely many $K$-Knuth moves.

The third and fourth relations have some important implications. The third rule implies that each $K$-Knuth equivalence class of words has infinitely many words of arbitrarily large length. Because Hecke insertion results in an increasing tableau, there are only finitely many tableaux into which words on an alphabet $[n]$ (words containing at least one of each letter from $\{1,2,...,n\}$) can be inserted. Hence there are finitely many equivalence classes on any alphabet $[n]$ with infinitely many words in each class. This is in contrast to the standard version, in which there are only finitely many words in each class.

The fourth rule implies that a letter can appear a different number of times in two equivalent words of the same length. For example, $121 \equiv 212$, but $1$ appears twice in the first word and once in the second.

We define the \textit{insertion class} of a word $w$ to be $\{w'\,:\,P(w')=P(w)\}$. We have the following result relating insertion class and $K$-Knuth equivalence class.

\begin{proposition}\cite[Theorem 6.2]{BuSa13}
If $w$ and $w'$ are in the same insertion class, then $w \equiv w'$.\end{proposition}

Hence Hecke insertion equivalence implies $K$-Knuth equivalence. The converse, however, is false: $K$-Knuth equivalent words may not be insertion equivalent. So unlike Knuth equivalence classes, $K$-Knuth equivalence classes may contain more than one insertion tableau.

\begin{example} \label{example:3124}
Let $w=1342$ and $w'=13422$.  Clearly $w$ and $w'$ are $K$-Knuth equivalent. Notice, however, that $$P(w)=\young(124,3) 
\hspace{0.3cm}
\mbox{and}
\hspace{0.3cm}
P(w')=\young(124,34).$$
\end{example}
We define
a recording tableau $Q(w)$ for the Hecke insertion of a word $w$,
in analogy to the recording tableau for Schensted insertion,
which allows one to uniquely recover an inserted word $w$ 
from the pair $(P(w),Q(w))$ via \textit{reverse Hecke insertion}. This definition is implicit in~\cite{BuKr06} and explained in detail in~\cite{PaPy14}.
We will not use this notion, but we note it for completeness.

\subsection{Reading Words}\label{sec:reading_word}

The Hecke insertion algorithm assigns
to each word an increasing tableau.
In this section, we describe a way
to associate to each increasing tableau
a certain set of words, called \emph{reading words}
for the tableau. A tableau can have many reading
words, and they will all be $K$-Knuth equivalent.

Let~$T$ be an increasing tableau.
The most commonly used reading word
for~$T$ is the \emph{row word} for~$T$,
written~$\row(T)$, which is obtained
by reading the entries of~$T$ 
from left to right along each row, 
starting from the bottom row
and moving upward.
Similarly, the \emph{column word} for~$T$,
written~$\col(T)$, is obtained
by reading the entries of~$T$ from
bottom to top along each column,
starting from the first column
and moving rightward.

More generally, a reading word of $T$ is any word $w$ listing the labels of the boxes of $T$ in any order for which the letter of $w$ corresponding to $\alpha$ appears before the letter of $w$ corresponding to $\beta$ whenever $\alpha$ is weakly southwest of $\beta$ in $T$.

\begin{example}
If 
\[
T=\young(1345,245,4)
\] 
then $\mathfrak{row}(T)=42451345$ and $\col(T)=42143545$. 
Two more valid reading words for $T$ are $42145345$ and $42413545$
\end{example}

\begin{proposition}\cite[Lemma 5.4]{BuSa13}
If~$w$ and $v$ are two reading words of an increasing tableau~$T$, then $w\equiv v$.
\end{proposition}

Let $T$ and $T'$ be two increasing tableaux.  If $\row(T)\equiv\row(T')$, we say that $T$ is \textit{$K$-Knuth equivalent} to $T'$ and write $T\equiv T'$, 
extending the $K$-Knuth equivalence relation to
the set of increasing tableaux.

\subsection{$K$-Jeu de Taquin}
The classical jeu de taquin (jdt) algorithm 
defines an equivalence relation on standard skew tableaux. Recall that a tableau $T$ with $n$ boxes is \textit{standard} if the entries $1,2,\ldots,n$ appear exactly once.

\begin{example}
Of the fillings below, only the second is standard.
\[
\ytableausetup{boxsize=\boxsize}
\begin{ytableau}
1 & 2 & 3 & 5 \\
3 & 4 & 5 & 6\\
7 & 8
\end{ytableau}\qquad
\begin{ytableau}
1 & 2 & 3 & 5 \\
4 & 6 & 7 & 9\\
8 & 10
\end{ytableau}
\]
\end{example}

In this section, we give a $K$-theoretic extension
of jdt to increasing tableaux introduced in \cite{ThYo07}, $K$-jdt, closely following the exposition of~\cite{BuSa13}. The $K$-jdt algorithm gives an alternative method for testing tableaux equivalence.

\begin{definition}
We say two boxes $\alpha,\beta\in\N\times\N$ 
are \textit{neighbors}
if $\alpha$ is directly above, below,
left of, or right of $\beta$.
Given a tableau $T$ and two entries $s$ and $s'$
of~$T$, define a new tableau $\swap_{s,s'}(T)$ of the same shape by
\[
\swap_{s,s'}(T):\alpha\mapsto
\begin{cases}
 s' &\text{if } T(\alpha)=s \text{ and } T(\beta)=s' 
 \text{ for some neighbor } \beta \text{ of } \alpha; \\ 
 s &\text{if } T(\alpha)=s' \text{ and } T(\beta)=s \text{ for some neighbor }\beta\text{ of } \alpha; \\ 
 T(\alpha) &\text{otherwise}.
\end{cases}
\]
\end{definition}

To define $K$-jdt, we allow the label $\bullet$ in addition to labels in $\mathbb{N}$. 
 
\begin{example}
Two examples of swaps may be found below.
\[
\young(12\bullet 4,\bullet 3,3)
\quad\underrightarrow{{\swap}_{\bullet,3}}\quad
\young(12\bullet 4,3\bullet ,\bullet)\]
\[\young( 124,34\bullet,\bullet)
\quad\underrightarrow{{\swap}_{\bullet,4}}\quad
\young(12\bullet,3\bullet 4,\bullet)
\]
\end{example}

A complete $K$-jdt move or slide consists of choosing an initial set of empty boxes to mark with $\bullet$ followed by a series of the swaps described above. Boxes marked with $\bullet$ begin either all weakly northwest of $T$ or all weakly southeast of $T$ and will move across the $T$ as swaps are performed during $K$-jdt. Each K-jdt move will be categorized as either a forward slide (where the boxes with $\bullet$ begin weakly northwest of $T$) or a reverse slide (where the boxes with $\bullet$ begin weakly southeast of $T$). We first define forward slides.

\begin{definition}[\cite{BuSa13}]
Let $T$ be an increasing tableau of shape $\nu/\lambda$ on letters in the interval $[a,b]$. Let $C\subset\lambda$ be a subset of the maximally southeast boxes, and mark the boxes in $C$ with $\bullet$. (So all boxes marked with $\bullet$ are weakly northwest of all boxes of $T$.)
We define the \textit{forward slide of $T$} starting from $C$ to be the composition of swaps 

$$\swap_{b,\bullet}(\ldots(\swap_{a+2,\bullet}(\swap_{a+1,\bullet}(\swap_{a,\bullet}(T))))\ldots).$$

Similarly, let $\widehat{C}\subset\N\times\N\setminus\nu$ be a subset of the maximally northwest boxes of 
$\N\times\N\setminus\nu$, and mark the boxes of $\widehat{C}$ with $\bullet$. (So all boxes marked with $\bullet$ are weakly southeast of all boxes in $T$.) Then the \textit{reverse slide of $T$} starting from $\widehat{C}$ is 
the composition of swaps 
$$\swap_{a,\bullet}(\swap_{a+1,\bullet}(\ldots(\swap_{b-1,\bullet}(\swap_{b,\bullet}(T))))\ldots).$$
\end{definition}

\begin{example}
The examples below give one complete forward $K$-jdt slide
and one complete reverse $K$-jdt slide, showing the sequence
of swaps performed during the slide.
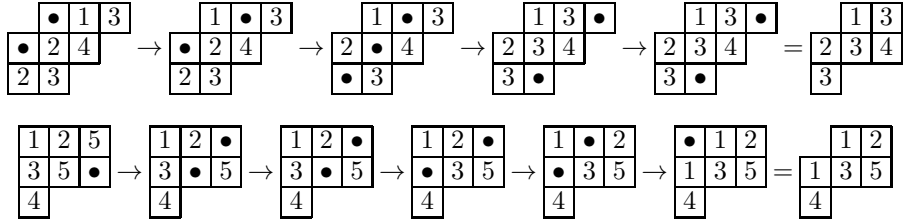
\[
\ytableausetup{boxsize=0.39cm}
\begin{ytableau}
\none & \bullet & 1 & 3 \\
\bullet & 2 & 4 \\
2 & 3
\end{ytableau}\rightarrow 
\begin{ytableau}
\none & 1 & \bullet & 3 \\
\bullet & 2 & 4 \\
2 & 3
\end{ytableau}\rightarrow 
\begin{ytableau}
\none & 1 & \bullet & 3 \\
2 &\bullet &  4 \\
\bullet & 3
\end{ytableau}\rightarrow 
\begin{ytableau}
\none & 1 & 3 & \bullet  \\
2 & 3 &  4 \\
3 & \bullet 
\end{ytableau}\rightarrow 
\begin{ytableau}
\none & 1 & 3 & \bullet  \\
2 & 3 &  4 \\
3 & \bullet 
\end{ytableau}=
\begin{ytableau}
\none & 1 & 3  \\
2 & 3 &  4 \\
3 
\end{ytableau}
\]

\hspace{2cm}
\[
\ytableausetup{boxsize=0.39cm}
\begin{ytableau}
1 & 2 & 5 \\
3 & 5 & \bullet \\
4
\end{ytableau}\rightarrow 
\begin{ytableau}
1 & 2 & \bullet \\
3 & \bullet & 5 \\
4
\end{ytableau}\rightarrow 
\begin{ytableau}
1 & 2 & \bullet \\
3 & \bullet & 5 \\
4
\end{ytableau}\rightarrow 
\begin{ytableau}
1 & 2 & \bullet \\
\bullet & 3 & 5 \\
4
\end{ytableau}\rightarrow 
\begin{ytableau}
1 & \bullet & 2 \\
\bullet & 3 & 5 \\
4
\end{ytableau}\rightarrow 
\begin{ytableau}
\bullet & 1 & 2 \\
1 & 3 & 5 \\
4
\end{ytableau}=
\begin{ytableau}
\none & 1 & 2 \\
1 & 3 & 5 \\
4
\end{ytableau}
\]
\end{example}

It is apparent from the definitions of $\swap$
that tableaux resulting from forward and reverse $K$-jdt slides remain increasing along rows and columns.
It is also apparent that, by design, forward and reverse moves are inverses. 
Furthermore, one can use forward moves to transform a skew shape into a straight shape and reverse moves to do the opposite. 

In the same way that the $K$-Knuth moves
give an equivalence relation on words,
$K$-jdt slides give an equivalence relation on tableaux.

\begin{definition}[\cite{BuSa13}]
We say that two increasing tableaux $S$ and $T$ are \textit{$K$-jeu de taquin equivalent} if $S$ can be obtained by applying a sequence of forward and reverse $K$-jeu de taquin slides to $T$ .
\end{definition}

The importance of $K$-jdt equivalence lies in the following theorem, proved in~\cite{BuSa13}.

\begin{theorem} \cite[Theorem 6.2]{BuSa13}
For tableaux $T$ and $T'$, $\row(T)\equiv\row(T')$ if and only if $T$ and $T'$ are $K$-jdt equivalent.
\end{theorem}

Therefore, $K$-jdt equivalence of tableaux is the same as $K$-Knuth equivalence of words.

\subsection{Unique Rectification Targets}
Applying forward $K$-jdt slides to a skew tableau~$T$ will eventually result in a straight tableau called a \textit{$K$-rectification} of~$T$. The \emph{rectification order} is the choice
of the placements of~$\bullet$'s for each
forward $K$-jdt slide.
In contrast to the classical theory of jdt,
different rectification orders may result
in different $K$-rectifications. 
In other words, varying the initial placements of $\bullet$'s
during $K$-jdt slides may result in different straight tableaux.

\begin{example}
Here is an example of how different rectification orders may produce different $K$-rectifications.
The tableau
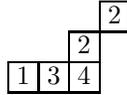
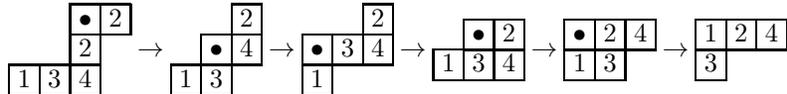
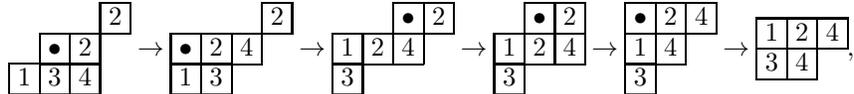
\[
\ytableausetup{boxsize=0.39cm}
\begin{ytableau}
\none & \none & \none & 2 \\
\none & \none & 2 \\
1 & 3 & 4
\end{ytableau}
\]
has the rectifications
\hspace{2cm}
\[
\begin{ytableau}
\none & \none & \bullet & 2 \\
\none & \none & 2 \\
1 & 3 & 4
\end{ytableau}\rightarrow 
\begin{ytableau}
\none & \none & 2 \\
\none & \bullet & 4 \\
1 & 3 
\end{ytableau}\rightarrow 
\begin{ytableau}
\none & \none & 2 \\
\bullet & 3 & 4 \\
1 
\end{ytableau}\rightarrow 
\begin{ytableau}
\none & \bullet & 2 \\
1 & 3 & 4 \\
\end{ytableau}\rightarrow 
\begin{ytableau}
\bullet & 2 & 4\\
1 & 3 \\
\end{ytableau}\rightarrow 
\begin{ytableau}
1 & 2 & 4\\
3 \\
\end{ytableau}
\]
and
\[
\ytableausetup{boxsize=0.39cm}
\begin{ytableau}
\none & \none & \none & 2 \\
\none & \bullet & 2 \\
1 & 3 & 4
\end{ytableau}\rightarrow 
\begin{ytableau}
\none & \none & \none & 2 \\
\bullet & 2 & 4 \\
1 & 3
\end{ytableau}\rightarrow 
\begin{ytableau}
\none & \none & \bullet & 2 \\
1 & 2 & 4 \\
3
\end{ytableau}\rightarrow 
\begin{ytableau}
\none & \bullet & 2 \\
1 & 2 & 4 \\
3
\end{ytableau}\rightarrow 
\begin{ytableau}
\bullet & 2 & 4\\
1 & 4 \\
3
\end{ytableau}\rightarrow 
\begin{ytableau}
1 & 2 & 4\\
3 & 4 
\end{ytableau},\]
resulting in different tableaux. Note that these tableaux are the same as in Example 2.11.
\end{example}

In some instances the $K$-rectification
may be unique, motivating the following
definition.

\begin{definition}\cite[Definition 3.5]{BuSa13}
An increasing tableau $U$ is a \textit{unique rectification target} (URT) if, for every increasing tableau $T$ that has $U$ as a rectification, $U$ is the only rectification of $T$.
\end{definition}

Equivalently, an increasing tableau is a URT if it is the only tableau in its $K$-Knuth equivalence class. 
The literature gives several classes of URTs,
which we summarize below.

\begin{definition}
A \textit{minimal} tableau is a tableau in which each box is filled with the smallest positive integer that will make the filling a valid increasing tableau.
\end{definition}

\begin{example}
The following tableau is minimal of shape $(4,3,3,1)$:
\begin{gather*}
\begin{ytableau} 
1 & 2 & 3 & 4\\ 
2 & 3 & 4\\ 
3 & 4 & 5\\ 
4\\
\end{ytableau}.
\end{gather*}
\end{example}

\begin{proposition}\cite[Corollary 4.7]{BuSa13}\label{prop:minimal_URT}
Every minimal tableau is a URT.
\end{proposition}

\begin{definition}
A \textit{superstandard} tableau is a standard tableau that fills the first row with $1,2, \ldots, \lambda_1$, the second row with $\lambda_1 +1, \lambda_1 +2, \ldots, \lambda_1+\lambda_2$, etc., where $\lambda_i$ is the length of the $i^{\text{th}}$ row of the tableau. 
\end{definition}

\begin{example}
The following tableau is superstandard of shape $(4,3,3,1)$:
\begin{gather*}
\begin{ytableau} 
1 & 2 & 3 & 4\\ 
5 & 6 & 7\\ 
8 & 9 & 10\\ 
11\\
\end{ytableau}.
\end{gather*}
\end{example}

\begin{proposition}
\label{superstandardURT}
Every superstandard tableau is a URT.
\end{proposition}
Proposition \ref{superstandardURT} is a corollary of~\cite[Theorem 3.7]{ThYo10};
it will also follow from Theorem~\ref{theorem1}.

Buch and Samuel proved in~\cite{BuSa13} that certain URTs can be added to minimal hooks to generate new URTs. We introduce this result with a few preliminary definitions.

\begin{definition}
A \textit{fat hook} is a partition of the form $(a^b, c^d)$, where $a,b,c,d$ are nonnegative integers with $a \geq c$.
\end{definition}

\begin{example}
The partition below is a fat hook of shape $(4^2, 2^3)=(4,4,2,2,2)$.
\begin{center}
\yng(4,4,2,2,2)
\end{center}
\end{example}

\begin{definition}
Let $M_{\lambda}$ be the minimal increasing tableau corresponding to a fat hook $\lambda = (a^b, c^d)$, and let $U$ be an increasing tableau. We say $U$ \textit{fits in the corner of} $M_{\lambda}$ if $U$ has at most $d$ rows, at most $a-c$ columns, and all integers contained in $U$ are strictly larger than all integers contained in $M_{\lambda}$. 
\end{definition}

In other words, $U$ fits in the corner of $M_{\lambda}$ if the entries of $U$ are strictly greater than the entries of $M_{\lambda}$ and if positioning $U$ in the corner of the hook results in an increasing tableau.

\begin{theorem}\cite[Theorem 6.9]{BuSa13}
Let $\lambda$ be a fat hook, and let $U$ be any unique rectification target that fits in the corner of $M_{\lambda}$. Then $M_{\lambda} \cup U$ is a unique rectification target.
\end{theorem}

\begin{example}
\label{fathookURT}
We may conclude that the tableau
\[
T = \young(1234,2345,367,47,5)
\]
is a URT because $T=M_{(4^2,1^3)}\cup U$, where
\[
M_{(4^2,1^3)} = \young(1234,2345,3,4,5)\qquad\text{and}\qquad
U = \young(67,7).
\]
We know that $U$ is a URT because its standardization is minimal.
\end{example}

\subsection{$K$-Knuth invariants} \label{section:invariants}
Now that we have the notion of an equivalence class of tableaux, we will provide several invariants under the $K$-Knuth equivalence relation. These will help prove results concerning the relations between tableaux in equivalence classes. A comprehensive list will be provided at the end of the section.

\begin{definition}
For a word $w$, let $\lis(w)$ (resp. $\lds(w)$) denote the length of the longest strictly increasing (resp. decreasing) subsequence of $w$.
\end{definition}

If~$w$ is the row word of a tableau~$T$ then $\lis(w)$ is the length of the first row of~$T$
and $\lds(w)$ is the length of the first column of~$T$.

\begin{example}
We will use the reading word of the tableau $T$ from Example \ref{fathookURT} to illustrate this concept. We see that if $w = \row(T) = 54736723451234$, then $\lis(w) = 4$, and $\lds(w) = 5$.
\end{example}

\begin{proposition}\cite[Theorem 6.1]{ThYo07}
\label{lis_lds_equality}
If $w_1 \equiv w_2$ then $\lis(w_1) = \lis(w_2)$ and $\lds(w_1) = \lds(w_2)$. 
\end{proposition}

The above equalities follow easily from the $K$-Knuth equivalence relations.

\begin{theorem}\cite[Theorem 1.3]{ThYo08}
\label{row_col_lis_lds}
For any word $w$, the size of the first row and first column of $P(w)$ are given by $\lis(w)$ and $\lds(w)$, respectively.
\end{theorem}

\begin{definition}
For a word \textbf{$w$}, let \textbf{$w|_{[a,b]}$} 
denote the word obtained from \textbf{$w$} by deleting all integers not contained in the interval $[a,b]$. 
Likewise, let $T$ be an increasing tableau, not necessarily straight,
and let $T|_{[a,b]}$ denote the tableau obtained from~$T$ by removing all boxes with labels outside of~$[a,b]$.
\end{definition}

\begin{proposition}\cite[Lemma 5.5]{BuSa13}
\label{restrictedalpha}
Let $[a,b]$ be an integer interval.
\begin{enumerate}
\item Let \textbf{$w_1$} and \textbf{$w_2$} be $K$-Knuth equivalent words.  Then \textbf{$w_1|_{[a,b]}$} and \textbf{$w_2|_{[a,b]}$} are $K$-Knuth equivalent words.
\item Let $T_1$ and $T_2$ be $K$-Knuth equivalent (possibly skew) tableaux. Then $T_1|_{[a,b]}\equiv T_2|_{[a,b]}$.
\end{enumerate}

\end{proposition}

Proposition \ref{restrictedalpha} is illustrated in the following example. 

\begin{example}
We have that $T_1 \equiv T_2$ for
\[
T_1 = \young(12345,23,3,5)\qquad
T_2 = \young(12345,235,35,5).
\]
Therefore $T_1|_{[3,5]} \equiv T_2|_{[3,5]}$,
so that
\[
\begin{ytableau}
\none & \none & 3 & 4 & 5\\
\none & 3 \\
3 \\
5
\end{ytableau}
\equiv 
\begin{ytableau}
\none & \none & 3 & 4 & 5\\
\none & 3 & 5 \\
3 & 5 \\
5
\end{ytableau}.
\]
\end{example}

\begin{definition}
Let $T$ be a straight tableau. The \emph{outer hook}
of $T$ is the subtableaux of $T$ consisting
of the first row and the first column.
\end{definition}
\begin{example}
The outer hook of the tableau below is shaded gray.
\[
\begin{ytableau}
*(lightgray) 1 & *(lightgray) 2 & *(lightgray) 4 & *(lightgray) 5\\
*(lightgray) 3 & 4 & 8\\
*(lightgray) 6 & 7
\end{ytableau}
\]
\end{example}
Our third invariant for $K$-Knuth classes is the outer hook.
\begin{proposition}\label{prop:outer_hook}
Let $T$ and~$T'$ be tableaux such that $T\equiv T'$.
Then $T$ and~$T'$ have the same outer hook.
\end{proposition}
\begin{proof}
Without loss of generality, assume $T$ and $T'$ are initial tableaux. Suppose $T$ and~$T'$ have letters in the alphabet~$[n]$.
Consider the two tableau sequences
\begin{gather*}
T|_{[1]},\,T|_{[2]},\ldots,\,T|_{[n]}\\
T'|_{[1]},\,T'|_{[2]},\ldots,\,T'|_{[n]}.
\end{gather*}
We have already seen that $T|_{[i]}\equiv T'|_{[i]}$ because $T\equiv T'$.
Now proceed by induction on~$n$.
The tableaux $T|_{[1]}$ and $T'|_{[1]}$
are equivalent and contain one box,
meaning $T|_{[1]} = T'|_{[1]}$.

Assume $T|_{[i]}$ and $T'|_{[i]}$
have the same outer hook. The outer hook of $T|_{[i]}$ (resp.~$T'|_{[i]}$)
differs from the outer hook of $T|_{[i+1]}$ (resp.~$T'|_{[i+1]}$)
by the addition of at most two boxes,
which must be labeled with $i+1$. By Proposition~\ref{lis_lds_equality} and Theorem~\ref{row_col_lis_lds},  $T|_{[i+1]}$ and $T'|_{[i+1]}$ must have the same number of rows and columns since $T|_{[i+1]}\equiv T'|_{[i+1]}$, so $T|_{[i+1]}$ and $T'|_{[i+1]}$ must have the same outer hook.
\end{proof}

Our fourth invariant is simple and involves the transpose
of a tableau:
if $T_1 \equiv T_2$, then $T_1^t \equiv T_2^t$.
Invariance under the transpose follows from the fact that
if $T_1$ and $T_2$ are $K$-Knuth equivalent then they
are $K$-jdt equivalent, and any sequence of $K$-jdt
moves connecting the two may be applied to their transposes.

\begin{example}
We have that 
\[
\young(124,3) \equiv \young(124,34)
\]
implying that 
\[
\young(13,2,4) \equiv \young(13,24,4)  \ .
\]
\end{example}

Another invariant is the \textit{Hecke permutation}, which was defined in~\cite{BuKr06} to provide a coarser equivalence than $K$-Knuth equivalence.
Define $\Sigma$ to be the group of bijective maps 
\[
\{w: \N \rightarrow \N \mid w(x) = x \text{ for all but finitely many } x \in \N\}.
\]
The adjacent transpositions $s_i=(i,i+1)\in\Sigma$
generate $\Sigma$ and give the group a natural 
presentation as a Coxeter group.
We will use this Coxeter group structure
to define a new product on $\Sigma$
that makes $\Sigma$ into a monoid.

Given a permutation~$u$, let $\ell(u)$
denote the shortest length of a factorization
$u=s_{i_1}\cdots s_{i_k}$ of $u$ as a product
of the $s_i$~transpositions.
Equivalently, $\ell(u)$ is the number
of inversions of~$u$, that is,
the number of pairs $i < j$ such that
$u(i) > u(j)$ (see \cite[Sec. 1.6 Exercise 2]{Hump92}).
\begin{definition}
Let $u\in\Sigma$ be a permutation. 
The \textit{Hecke product} of $u$ and a transposition~$s_i$
is defined by
\[
u \cdot s_i = \begin{cases} us_i & \mbox{ if } \ell(us_i) > \ell(u); \\
								u & \mbox{ otherwise. } \\
                 \end{cases}
\]
Given a second permutation 
$v =  s_{i_1} \cdot s_{i_2} \cdots s_{i_l}\in \Sigma$, 
the Hecke product of $u$ and $v$ is defined by
\[
u \cdot v = u \cdot s_{i_1} \cdot s_{i_2} \cdots s_{i_l},
\]
multiplying from left to right.
\end{definition}

The Hecke product is associative and gives
a monoidal structure on $\Sigma$, 
allowing us to introduce the following concept.

\begin{definition}
The \textit{Hecke permutation} of an increasing tableau $T$ is $w(T) = w(a) = s_{a_1} \cdot s_{a_2} \cdots s_{a_k}$, where $a = a_1\cdots a_k$ is a reading word of $T$. 
\end{definition}

\begin{proposition} \label{permutation}
The Hecke permutation $w(T)$ of an increasing tableau is invariant under $K$-Knuth moves.
\end{proposition}
Proposition~\ref{permutation} is equivalent 
to Corollary 6.5 in \cite{BuSa13}, using the fact that $K$-jdt equivalence implies $K$-Knuth equivalence. 

Having the same Hecke permutation is a necessary but not sufficient condition for two tableaux to appear in the same $K$-Knuth class,
meaning that the number of $K$-Knuth equivalence classes on~$[n]$ is at least as large as $S_{n+1}$, the symmetric group on $n+1$ elements. Hence there are a minimum of $(n+1)!$ $K$-Knuth classes
of tableaux on~$[n]$ letters.
Proposition~\ref{tableauxon2n}
will make use of this fact.
\begin{example}
The Hecke permutation for the word $21231$
is given by
\[
1\mapsto3,\quad2\mapsto2,\quad3\mapsto4,\quad4\mapsto1.
\]
\end{example}

In summary, the following are invariant under the $K$-Knuth 
equivalence relations:
\begin{enumerate}
\item the length of the longest strictly increasing (or decreasing) subword of a word,
\item the restriction of a word or a tableau to an interval subalphabet, up to $K$-Knuth equivalence,
\item the outer hook of a tableau,
\item the transpose of a tableau, up to $K$-Knuth equivalence, and
\item the Hecke permutation.
\end{enumerate}

\section{Algorithms} \label{section:algorithms}

This section deals with computational aspects of the $K$-Knuth equivalence relation. 
We will describe an algorithm to solve the 
following two problems:
\begin{enumerate}
\item Determine if two words are $K$-Knuth equivalent.
\item Compute all $K$-Knuth classes of tableaux
on~$[n]$.
\end{enumerate}
Specifically, the algorithm will solve the second problem.
Given that solution,
we can solve the first problem by computing the insertion tableaux of the two words.

Let $\mathbb{T}_n$ be the set of (not necessarily initial) increasing tableaux on~$[n]$, and let $N_n$ be its cardinality. 
We first remark that it is fairly easy to construct the set~$\mathbb{T}_n$.
Indeed, one can construct it recursively.
Given $\mathbb{T}_{n-1}$, one can check for each $T\in\mathbb{T}_{n-1}$ where one can add boxes with entry $n$ to $T$ to get $T'\in\mathbb{T}_n$.
Therefore, we will assume that we have the set~$\mathbb{T}_n$ at our disposal in the algorithm that follows.

We say that a pair $(a,b)$ of words is a \emph{primitive pair} if
\begin{enumerate}
	\item $a = p$, $b = pp$ for some letter $p$;
	\item $a = pqp$, $b = qpq$ for some letters $p\neq q$;
	\item $a = xzy$, $b = zxy$ for some letters $x < y < z$; or
	\item $a = yxz$, $b = yzx$ for some letters $x < y < z$.	
\end{enumerate}
Throughout the algorithm, we will maintain a set partition $\mathcal{P}$ of $\mathbb{T}_n$ 
and a queue~$Q$ that stores unordered pairs $\{T_1,T_2\}$ of tableaux in~$\mathbb{T}_n$. Let $\pend$ denote the set partition
$\mathcal P$ at the end of the algorithm.
We claim that $\pend$ gives the $K$-Knuth
equivalence classes of $\mathbb T_n$.
In what follows, we say the partition $\mathcal P$
\emph{joins} the tableaux $T_1$ and $T_2$
if $\mathcal P$ contains a set containing both
$T_1$ and $T_2$.

\begin{algorithm}[h]
  \caption{Algorithm for Computing all $K$-Knuth Classes  \label{algorithm}}
  \begin{algorithmic}[1]
  	\State Initialization: $\mathcal{P}\coloneqq\{\{T\} : T\in\mathbb{T}_n\}$; $Q$ empty.
    \ForAll{ $T\in\mathbb{T}_n$}\label{alg_line:loop1_start}
    	\ForAll{ primitive pair $(a,b)$}
      	\State Compute $T_1\coloneqq T\leftarrow a$ and $T_2\coloneqq T\leftarrow b$.
      	\If{$T_1$ and $T_2$ are not joined by $\mathcal{P}$}\label{alg_line:check_join1}
      		\State Merge the sets in $\mathcal{P}$ containing $T_1$ and $T_2$.\label{alg_line:merge1}
      		\State Insert the pair $\{T_1,T_2\}$ into $Q$.\label{alg_line:queue_insert1}
      	\EndIf
      \EndFor
    \EndFor\label{alg_line:loop1_end}
     \While{$Q$ is non-empty}\label{alg_line:loop2_start}
    	\State Remove a pair $\{U_1,U_2\}$ from $Q$.\label{alg_line:queue_remove}
    	\ForAll{$1\leq y\leq n$}
      	\State Compute $T_1\coloneqq U_1\leftarrow y$ and $T_2\coloneqq U_2\leftarrow y$.
      	\If{$T_1$ and $T_2$ are not joined by $\mathcal{P}$}
      		\State Merge the sets in $\mathcal{P}$ containing $T_1$ and $T_2$.
            
                \algstore{myalg}
\end{algorithmic}
\end{algorithm}

\begin{algorithm}                     
\begin{algorithmic} [h]                   
\algrestore{myalg}
            
      		\State Insert the pair $\{T_1,T_2\}$ into $Q$.\label{alg_line:queue_insert2}
      	\EndIf
      \EndFor
    \EndWhile\label{alg_line:loop2_end}
    \State Output $\mathcal{P}$.
\end{algorithmic}
\end{algorithm}


\begin{theorem}\label{thm:algorithm correctness}
$T_1$ and $T_2$ are $K$-Knuth equivalent if and only if they are joined by~$\pend$.
\end{theorem}

We need a few preliminary lemmas to prove Theorem~\ref{thm:algorithm correctness}. In the rest of the section, we let $\mathcal{Q}$ denote the set of pairs $\{T_1,T_2\}$ that have been inserted into $Q$:
$$\mathcal{Q}=\left\{\{T_1,T_2\} : \text{$\{T_1,T_2\}$ has been inserted into $Q$ (in line~\ref{alg_line:queue_insert1} or~\ref{alg_line:queue_insert2})}\right\}. $$

\begin{lemma}\label{lem:algorithm correctness 4}
If $\{T_1,T_2\}\in\mathcal{Q}$, then $T_1\equiv T_2$.
\end{lemma}
\begin{proof}
Assume that $\{T_1,T_2\}$ is the $k$th pair inserted into $Q$. We will proceed by induction on $k$. If $\{T_1,T_2\}$ is inserted into $Q$ in line~\ref{alg_line:queue_insert1}, then $T_1$ and $T_2$ are obtained by inserting $a$ and $b$ into the same tableau respectively for some primitive pair $(a,b)$. Since $a\equiv b$, $T_1\equiv T_2$. If $\{T_1,T_2\}$ is inserted into $Q$ in line~\ref{alg_line:queue_insert2}, then $T_1$ and $T_2$ are obtained by inserting the same letter $y$ into $U_1$ and $U_2$ respectively for some pair $\{U_1,U_2\}$ previously inserted into $Q$. By induction hypothesis, $U_1\equiv U_2$, so $T_1\equiv T_2$.
\end{proof}

\begin{lemma}\label{lem:algorithm correctness 2}
Suppose $\pend$ joins the tableaux $T_1$ and $T_2$. Then, there exists a sequence $T_1 = U_0,U_1,\ldots,U_r = T_2$ of tableaux such that $\{U_i,U_{i+1}\}\in\mathcal{Q}$ for each $i$.
\end{lemma}
\begin{proof}
Assume that $T_1$ and $T_2$ are joined by $\mathcal P$ after the $k$th merge but not before. We will proceed by induction on $k$. The case $k=0$ is trivial. Suppose $k > 0$. Let $S_1$ and $S_2$ be the sets containing $T_1$ and $T_2$, respectively, before the $k$th merge. 
By assumption, the sets $S_1$ and $S_2$ merge at the $k$th merge. 
After that merge, we insert into $Q$ the pair $\{V_1,V_2\}$, for some $V_1\in S_1$ and $V_2\in S_2$, so by definition $\{V_1,V_2\}\in\mathcal{Q}$. We know from the induction hypothesis that there is a chain of pairs in $\mathcal{Q}$ connecting $T_1$ and $V_1$ as well as $T_2$ and $V_2$. The result then follows.
\end{proof}

\begin{lemma}\label{lem:algorithm correctness 1}
Fix $1\leq y\leq n$. Let $T_1$ and $T_2$ be two tableaux and let $T_i' = T_i\leftarrow y$. If $\pend$ joins $T_1$ and $T_2$,
then it joins $T_1'$ and $T_2'$.
\end{lemma}
\begin{proof}
By Lemma \ref{lem:algorithm correctness 2}, it suffices to consider the case where $\{T_1,T_2\}\in\mathcal{Q}$. In this case, 
$\{T_1,T_2\}$ is eventually removed from $Q$ in line~\ref{alg_line:queue_remove}. If $T_1'$ and $T_2'$ are joined by $\mathcal{P}$ at this point, then the assertion holds. Otherwise, we will merge the sets containing $T_1'$ and $T_2'$.
\end{proof}

\begin{lemma}\label{lem:algorithm correctness 3}
Let $T\in\mathbb{T}_n$ and let $(a,b)$ be a primitive pair. Set $T_1 = T\leftarrow a$ and $T_2 = T\leftarrow b$. Then, $T_1$ and $T_2$ are joined by $\pend$.
\end{lemma}
\begin{proof}
Eventually, we will check at line~\ref{alg_line:check_join1} whether $\mathcal P$ joins $T_1$ and $T_2$. If it does, then we are fine. Otherwise, we will merge the sets containing them on line~\ref{alg_line:merge1}.
\end{proof}

\begin{proof}[Proof of Theorem \ref{thm:algorithm correctness}]
The ``if'' direction follows from Lemma \ref{lem:algorithm correctness 4} and \ref{lem:algorithm correctness 2}.

For the ``only if'' direction, we start with the special case where there exist words $w_1$ and $w_2$ that differ by one $K$-Knuth move and such that $P(w_1) = T_1$ and $P(w_2) = T_2$. Write $w_1 = uav$ and $w_2 = ubv$ where $(a,b)$ is a primitive pair, so that by Lemma~\ref{lem:algorithm correctness 3} the tableaux $P(ua)$ and $P(ub)$ are joined by $\pend$. Applying Lemma \ref{lem:algorithm correctness 1} multiple times, we conclude that $\pend$ joins $T_1$ and $T_2$.

For the general case, there is a sequence $T_1 = U_0,U_1,\ldots,U_r = T_2$ of tableaux such that for each $i$ there exist two words differing by one $K$-Knuth move and inserting into $U_i$ and $U_{i+1}$ respectively. By the previous case, $\pend$ joins $U_i$ and $U_{i+1}$, so $\pend$ joins all $U_i$'s. 
In particular, $\pend$ joins $T_1$ and $T_2$.
\end{proof}

Having shown the correctness of the algorithm, 
we will now briefly analyze the runtime.
Of the operations performed during the algorithm,
the three we will focus on
are the following:
inserting a word into a tableau,
determining if two tableaux are joined by the partition,
and merging two sets
of the partition. These are the nontrivial
operations of the algorithm,
and their runtime depends on implementation-specific
details with which we will not concern ourselves.
The reader may supply
his or her own runtime estimates for each operation.

The algorithm consists of two successive loops,
comprising lines \ref{alg_line:loop1_start} -- \ref{alg_line:loop1_end} and \ref{alg_line:loop2_start} -- \ref{alg_line:loop2_end},
and each of the three operations is performed
at most twice during each loop.
Letting $N_n$ denote the cardinality of~$\mathbb T_n$,
we see that the first loop runs~$O(n^3 N_n)$ times
and the second loop runs at most $n(N_n-1)$ times,
since every element inserted into~$Q$
represents a merge of two sets in the partition.
We therefore have the upper bound of $O(n^3 N_n)$
for the number of insertion, determination
and merge operations in the algorithm.

\section{Length of Intermediate Words}
If $w\equiv w'$, then there exists a sequence $w = w_0,w_1,\ldots,w_r = w'$ of words such that $w_i$ and $w_{i+1}$ differ by one $K$-Knuth move. It is natural to ask whether it is always possible to find such a sequence where the intermediate words $w_i$ have length at most that of the longer of $w$ and $w'$. 
Surprisingly, the answer is no: one can check by computer that $4235124\equiv 4523124$ but the two words cannot be connected by words of length at most $7$. 
However, it is possible to give a weaker upper bound in terms of the size of the alphabet.

\begin{definition} \label{definition:length}
Let $w$ and $w'$ be words and let $k$ be a positive integer. 
We say that $w$ and $w'$ are \emph{equivalent through words
of length~$k$}, written $w\overset{k}{\equiv}w'$, if there exists a sequence $w = w_0,w_1,\cdots,w_r = w'$ of words such that $w_i$ and $w_{i+1}$ differ by one $K$-Knuth move, and each word $w_i$ has length at most $k$.
\end{definition}

We will prove the following result.

\begin{theorem}\label{thm:bound on length}
Suppose $T_1\equiv T_2$ are tableaux on~$[n]$. Let $N = n(n+1)(n+2)/3 + 3$. Then $\row(T_1)\overset{N}{\equiv}\row(T_2)$.
\end{theorem}

Computer evidence suggests that the bound in Theorem~\ref{thm:bound on length}
can be tightened to the largest size of a tableau
in the $K$-Knuth equivalence class, where the
\emph{size} of a tableau~$T$ of shape~$\lambda$
is the number of boxes of $\lambda$.{}

\begin{conjecture}\label{conj:bound on length}
Let $T$ and $T'$ be two tableaux with $T\equiv T'$,
and let $k$~be the largest size of a tableau
$K$-Knuth equivalent to $T$ or $T'$.
Then $\row(T)\stackrel{k}{\equiv}\row(T')$.
\end{conjecture}

Conjecture \ref{conj:bound on length} has been verified for tableaux on~$[n]$ with $n\leq 5$.

\subsection{Proof of Theorem \ref{thm:bound on length}}
We will use the following lemma, which concerns $K$-Knuth equivalence within an insertion class, to prove Theorem~\ref{thm:bound on length}.
Let $|w|$ denote the number of letters in a word~$w$. 

\begin{lemma}\label{lem:equiv within class}
If $w$ is a word and $P(w) = T$, then $w\overset{|w|}{\equiv}\row(T)$.
\end{lemma}
We defer the proof of Lemma~\ref{lem:equiv within class} to Section \ref{section:proof of equiv within class} because it is fairly technical.

Assuming Lemma~\ref{lem:equiv within class},
our first step toward the theorem is the reduction to the case where there exist words $w_1$ and $w_2$ such that $w_2$ differs from $w_1$ by one $K$-Knuth move and such that $P(w_1) = T_1$ and $P(w_2) = T_2$.

Suppose the result had been shown for the special case above. Now let $T_1$ and $T_2$ be any two $K$-Knuth equivalent tableaux on~$[n]$, let $\row(T_1)=w_0,w_1,w_2,\ldots,w_r = \row(T_2)$ be a sequence of equivalent words differing by one $K$-Knuth move, and let $U_i=P(w_i)$. By the result for the special case $\row(U_i)\overset{N}{\equiv}\row(U_{i+1})$, and
the general case follows.

To prove the result for the special case, we can just construct words $w_1'$ and $w_2'$ with $|w_1'|,|w_2'|\leq N$ such that $w_2'$ differs from $w_1'$ by a single $K$-Knuth move, $P(w_1') = T_1$, and $P(w_2') = T_2$. Indeed, by Lemma \ref{lem:equiv within class}, we have $w_1'\overset{N}{\equiv}\row(T_1)$ and $w_2'\overset{N}{\equiv}\row(T_2)$, and the result then follows.

The construction of the words $w_1'$ and $w_2'$ relies on the following observation: if $t$ is a letter of a word $w = u_1tu_2$ for which $P(u_1t) = P(u_1)$, i.e., if $t$ ``does nothing'' in the insertion of $w$, then $P(u_1tu_2) = P(u_1u_2)$. More precisely, write $w_1 = uav$ and $w_2 = ubv$ where $(a,b)$ is a primitive pair,
following the definition preceding the algorithm of Section~3.
Let $u'$ (resp. $v'$) be the word obtained by deleting all letters in $u$ (resp. $v$) which ``do nothing'' in the insertion of both $w_1$ and $w_2$. The words $w_1' = u'av'$ and $w_2' = u'bv'$ then satisfy $P(w_1') = T_1$ and $P(w_2') = T_2$. To finish the proof of the theorem, we will need the following upper bound on the number of indices which ``do something'' in the insertion.

\begin{lemma}\label{lem:no of relevant letters}
If $w=w_1w_2\cdots w_k$
is a word on~$[n]$,
then there are at most
$n(n+1)(n+2)/6$
indices~$r$ such that
$P(w_1\cdots w_{r-1})\not=P(w_1\cdots w_{r-1}w_r).$
\end{lemma}

We may now conclude the proof of Theorem~\ref{thm:bound on length} using Lemma \ref{lem:no of relevant letters}. If we apply the lemma with $w = uav$, we see that the \emph{total} number of indices in $u$ and $v$ that ``do something'' in the insertion of $uav$ is at most $n(n+1)(n+2)/6$. The same is true with $uav$ replaced by $ubv$, so $|u'|+|v'|$ is at most $n(n+1)(n+2)/3$. We therefore get the result.

\begin{proof}[Proof of Lemma \ref{lem:no of relevant letters}]
Fix $i$ and $j$ and set $k=i+j-1$.
If $P(w_1\cdots w_{r-1})$ and $P(w_1\cdots w_r)$ have different $(i,j)$th entries, then the insertion of $w_r$ either creates a new entry at the $(i,j)$th position or decreases the $(i,j)$th entry. 
At the end the $(i,j)$th entry must be at least $k$, so there are at most $n-k+1$ indices in $w$ that change the $(i,j)$th entry.
For a fixed $k$, there are exactly $k$ pairs of $(i,j)$ with $k=i+j-1$, so the result follows from the identity
\[
\sum\limits_{k=1}^{n}k(n-k+1) = \frac{1}{6}n(n+1)(n+2). \qedhere
\]
\end{proof}

\subsection{Proof of Lemma \ref{lem:equiv within class}}\label{section:proof of equiv within class}
The proof of Lemma~\ref{lem:equiv within class} 
will consist of a careful analysis of the Hecke insertion
algorithm via a sequence of reductions.
In essence, computing an insertion tableau~$P(w)$ is the same as
making a sequence of $K$-Knuth moves to the word~$w$,
and none of these moves lengthens the word.

The first step of the proof
of Lemma~\ref{lem:equiv within class} is the reduction to the special case where $w = \row(T')y$ for some tableau $T'$ and some letter $y$. 
Assuming this case has been proved, write $w = w_1w_2\cdots w_k$. Let $T_i = P(w_1w_2\cdots w_i)$, $t_i = \row(T_i)$, and $u_i = t_iw_{i+1}\cdots w_k$. The assertion in the special case implies that $t_{i+1}\overset{i+1}{\equiv}t_iw_{i+1}$, 
so that $u_{i+1}\overset{k}{\equiv}u_i$. 
The general case then follows.

It will be useful to introduce the following terminology in order to further simplify the special case $w = \row(T')y$.
This terminology will only be used in this section.

\begin{definition}
A \emph{row} is a sequence $R=(r_1,\dots,r_m)$
of strictly increasing integers. The \emph{length}
of~$R$, denoted by $|R|$, is $m$.

Given two rows $R$ and~$S$, write $R>S$
if $S$ can be placed above $R$ to make a two-row
tableau; that is, $R>S$ if $|R|\leq|S|$ and $R_i>S_i$ 
for all~$i$. Write $R\geq S$ if $|R|\leq|S|$
and $R_i\geq S_i$ for all~$i$.

Given two words $w$ and $w'$,
write $w\overset{\bullet}{\equiv} w'$
if $w$ and $w'$ are equivalent through
words of length $\max(|w|,|w'|)$.
\end{definition}

The following lemma will allow us to bump
a letter from one row to the next
via sequences of $K$-Knuth moves
that do not increase the length of the word.
\begin{lemma} 
\label{lem:bump}
\begin{enumerate}
	\item\label{lem:bump_part1} Let $R$ be a row and $y$ a letter. Let $R'$ be the top row of the tableau $R\leftarrow y$ and let $x$ be the bottom row, taken to be empty if necessary. Then $Ry\overset{\bullet}{\equiv} xR'$.

	\item\label{lem:bump_part2} Let $S_1$, $S_2$, $R_1$, and $R_2$ be rows and let $x$ and $y$ be letters such that $R_1>S_1$, $|R_1|=|S_1|$, $R_2>S_2$, $\max(R_1) < y < \min(S_2)$, and $y < x < \min(R_2)$. Then,
\[
\ytableausetup{boxsize=2em}
\row\,\begin{ytableau}
S_1 & y & S_2 \\
R_1 & y & R_2 \\
x& \none & \none \\
\end{ytableau}
\overset{\bullet}{\equiv}
\row\,\begin{ytableau}
S_1 & y & S_2 \\
R_1 & x & R_2 \\
x& \none & \none \\
\end{ytableau}
\]

	\item\label{lem:bump_part3} Let $R$ and $S$ be rows. Let $R'$ be the top row of the tableau $R\leftarrow y$ and let $x$ be the bottom row, taken to be empty if necessary. Assume that $R>S$ and $R'\geq S$ hold but $R'>S$ does not hold. Then, $xRS\overset{\bullet}{\equiv} xR'S$, and hence $xRS\overset{\bullet}{\equiv} RyS$.

\end{enumerate}

\end{lemma}

We have technically only defined
$\row$ for tableaux that are increasing.
The meaning of~$\row$ in (\ref{lem:bump_part2}) for nonincreasing tableaux
is the obvious one: the concatenation of rows
proceeding from bottom to top.

Before we prove Lemma~\ref{lem:bump},
let us see how it implies
the special case $w = \row(T')y$
of Lemma~\ref{lem:equiv within class}.
Let $T=P(w)$ and decompose $T$
and $T'$ into rows:
$\row(T)=R_\ell R_{\ell-1}\dots R_1$ and
$\row(T')=R'_{\ell+1}R'_\ell\dots R'_1$.
Let $y_i$ be the (possibly empty) letter inserted
into $R_i$ during the Hecke insertion
algorithm for $T\leftarrow y$.
Letting $R'_0$ be empty,
it suffices to prove that
\begin{equation}\label{eq:bump}\tag{$\star$}
R_iy_iR'_{i-1}\overset{\bullet}{\equiv} y_{i+1}R'_iR'_{i-1}.
\end{equation}
Following the description of the Hecke
insertion algorithm given in Section~2
for the insertion of $y_i$ into~$R_i$,
Lemma~\ref{lem:bump}(\ref{lem:bump_part1})
implies that \eqref{eq:bump} holds
in the case where the $i$th insertion
results in a valid tableau,
and Lemma~\ref{lem:bump}(\ref{lem:bump_part3})
implies that \eqref{eq:bump} holds
in the case where the $i$th insertion 
does \emph{not} result in a valid tableau.

It remains to prove Lemma~\ref{lem:bump}.

\begin{proof}[Proof of Lemma~\ref{lem:bump}(\ref{lem:bump_part1})]
Let $R=(r_1,\dots,r_m)$. If $r_m<y$ then $R'=Ry$; if $r_m=y$ then $R'=R$.
Otherwise, let $r_i$ be the
smallest entry of~$R$ greater than $y$.
If $r_{i-1} < y$, 
then $r_1\cdots r_my\overset{\bullet}{\equiv}
r_1\cdots r_iyr_{i+1}\cdots r_m$,
by a sequence of $K$-Knuth moves
of the form $bca\equiv bac$ ($a<b<c$), and $r_1\cdots r_iy\overset{\bullet}{\equiv}r_ir_1\cdots r_{i-1}y$ by a sequence of $K$-Knuth moves
of the form $acb\equiv cab$ ($a<b<c$).
If $r_{i-1}=y$ then
\begin{align*}
r_1\cdots r_my&\overset{\bullet}{\equiv}
r_1\cdots r_{i-1}r_{i}yr_{i+1}\cdots r_m\\
&\overset{\bullet}{\equiv}
r_1\cdots r_iyr_ir_{i+1}\cdots r_m\\
&\overset{\bullet}{\equiv}
r_{i}r_1\cdots r_{i-2}yr_{i+1}\cdots r_m.\qedhere
\end{align*}
\end{proof}

\begin{proof}[Proof of Lemma~\ref{lem:bump}(\ref{lem:bump_part2})]
The proof will consist of reducing first to
the case where $S_2$ and $R_2$ are empty,
then to the case where, in addition,
$S_1$ and $R_1$ have length one.
(The case where $S_1$, $S_2$,
$R_1$, and $R_2$ are empty is trivial.)

To make the first reduction, observe that if
\[
xR_1yS_1y\overset\bullet\equiv xR_1xS_1y
\]
then 
\[
xR_1yS_1yR_2\overset\bullet\equiv xR_1xS_1yR_2.
\]
It now follows from (\ref{lem:bump_part1}),
by Hecke inserting in increasing order each element of $S_2$
into the one-row tableau whose row word is $S_1yR_2$,
that
$S_1yR_2S_2\overset\bullet\equiv R_2S_1yS_2$,
and hence that
\[
xR_1yR_2S_1yS_2\overset\bullet\equiv xR_1xR_2S_1yS_2.
\]

To make the second reduction,
let $R_1=(r_1,\dots,r_m)$,
let $S_1=(s_1,\dots,s_m)$,
and let $R_1'$ and $S_1'$
be the rows obtained from $R_1$
and $S_1$, respectively,
by removing the final entry.
If we assume the statement holds whenever $|R_1| = |S_1| = 1$, then we have $xr_mys_my\overset\bullet\equiv xr_mxs_my$.
Applying (\ref{lem:bump_part1})
to insert $xr_ms_mys_m$ into the tableau
with row word $R_1'S_1'$ gives the 
equivalences
\[
R_1'S_1'xr_mys_my\overset\bullet\equiv
R_1'S_1'xr_ms_mys_m\overset\bullet\equiv
xR_1yS_1y
\]
while inserting $xr_mxs_my$ into the same tableau gives
\[
R_1'S_1'xr_mxs_my\overset\bullet\equiv
xR_1xS_1y,
\]
so we get $xR_1yS_1y\overset\bullet\equiv xR_1xS_1y$, as needed.

Finally, the case where $S_2$ and $R_2$
are empty and $S_1$ and $R_1$ have length
one reduces by standardization
to the equivalence $42313\overset\bullet\equiv 42413$,
which the reader may verify.
\end{proof}

\begin{proof}[Proof of Lemma~\ref{lem:bump}(\ref{lem:bump_part3})]
Let $R'=(r_1',r_2',\dots,r_m')$ and
let $S=(s_1,\dots,s_\ell)$.
Let $r'_i$ be the letter of~$R'$ that was modified
during the insertion of $y$ into $R$,
so that $r'_i=y$, $r_i=x$, and $r'_j=r_j$ for $i\neq j$.
The hypothesis that $R'\geq S$
holds but $R>S$ does not hold means that
$s_i=r_i=y$.
Statement (\ref{lem:bump_part2}) implies that $xRS\overset\bullet\equiv xR'S$. By  (\ref{lem:bump_part1}), $xR'\overset\bullet\equiv Ry$, so $xRS\overset\bullet\equiv RyS$.
\end{proof}

\section{Right-Alignable Tableaux}
In this section, we give a new family of URTs called
\textit{right-alignable tableaux}. As a corollary, we will deduce that superstandard and rectangular tableaux are URTs.

Although we have considered so far only fillings of straight shapes, in this section we will think about
fillings of more general shapes, using the formulation described in Section 2.1. We will use the terminology \textit{filling} and \textit{increasing filling} instead of ``tableau'' and ``increasing tableau'' to emphasize that we are discussing more general shapes.

\begin{definition}
Let $\lambda=(\lambda_1,\dots,\lambda_\ell)$ 
be a straight shape and let $T:\lambda\to\N$
be a tableau of shape~$\lambda$.
The \emph{right alignment} 
of $\lambda$ is the shape 
$\lambda_R = \{(i,j)\,:\,(\lambda_1-\lambda_i)<j\leq\lambda_1\}$.
The \emph{right alignment} of~$T$
is the filling $T_R:\lambda_R\to\N$ defined by
$T_R(i,j) = T(i,j-\lambda_1+\lambda_i)$.

A tableau~$T$ is \emph{right-alignable}
if the right alignment~$T_R$ is an increasing filling.
\end{definition}
\begin{example}
The tableau
\ytableausetup{mathmode, boxsize=1.4em}
\begin{gather*}
T = \begin{ytableau} 
1 & 2 & 3 \\
3 & 5 \\
4
\end{ytableau}
\quad\text{ has right alignment} \ 
T_R = \begin{ytableau} 
1 & 2 & 3 \\
\none & 3 & 5 \\
\none & \none & 4
\end{ytableau}.
\end{gather*}
Hence $T$~is not right-alignable.
\end{example}

We will prove the following theorem in the next
three sections.
\begin{theorem} \label{theorem1}
Every right-alignable tableau is a URT.
\end{theorem}
The following two corollaries follow easily from 
Theorem~\ref{theorem1}.
\begin{corollary}
Every superstandard tableau is a URT.
\end{corollary}
\begin{corollary}\label{cor:rect_URT}
Every tableau of rectangular shape is a URT.
\end{corollary}

Theorem~\ref{theorem1} is by no means sharp:
there are URTs that are not right-alignable.
For example, a minimal tableau is right-alignable
if and only if it is rectangular.

\subsection{Hook Closure Properties}
In this section, we will generalize the idea of hook closure as defined in~\cite[Section 5]{BuSa13}, which we will call \emph{northeast-hook closure}. This will allow us to ensure boxes exist in certain positions. Recall that the $(i,j)$th entry of a tableau
lies in the $i$th row and $j$th column.

\begin{definition} \label{def:hookClosed}
\label{hook-closed}
A shape $\lambda$ is \emph{northwest-hook-closed} if
it is closed under forming northwest hooks:
whenever $x=(i_1,j_1)$ and $y=(i_2,j_2)$ are boxes of $\lambda$
such that $i_1\geq i_2$ and $j_1 \leq j_2$
then $\lambda$ contains the boxes 
$(r,j_1)$ for $i_1\geq r\geq i_2$ and
$(i_2,c)$ for $j_1 \leq c\leq j_2$.
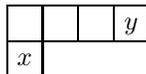
\ytableausetup{boxsize=\boxsize}
\[
\begin{ytableau}
{}& & &y \\
x&\none&\none&\none
\end{ytableau}
\]
We define \emph{northeast-hook-closed} and \emph{southeast-hook-closed} shapes in a similar way.
\end{definition}
One can define \emph{southwest-hook-closed} shapes as well, but our discussion will not concern this hook closure property.

\begin{example}
Of the shapes below, only the first is northwest-hook-closed,
only the second is northeast-hook-closed,
and only the third is southeast-hook-closed.
The fourth satisfies none of the three hook-closure properties defined.
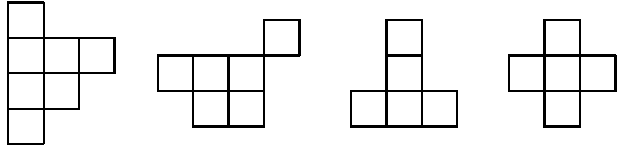
\[
\begin{ytableau}
\none& & \none \\
\none& & & \\
\none& & \\
\none& & \none
\end{ytableau}\,\,
\begin{ytableau}
\none & \none & \none & \none & \\
\none &&&&\none \\
\none & \none & & & \none
\end{ytableau}\qquad
\begin{ytableau}
\none & \\
\none & \\
&&
\end{ytableau}\qquad
\begin{ytableau}
\none & & \none \\
& & \\
\none& & \none
\end{ytableau}
\]
Proposition~\ref{prop:f_compat_invar} will concern shapes that are
northwest- and southeast-hook closed. 
Such shapes are the reflections
of skew shapes across a vertical axis
like the examples shown below.
\[
\begin{ytableau}
{} & & &\\
\none & & &\\
\none &\none& &\\
\none&\none&\none&
\end{ytableau}\qquad
\begin{ytableau}
{} & &\none&\none&\none\\
\none & & &\none\\
\none&\none&\none&
\end{ytableau}
\]
\end{example}

Lemma~\ref{lem:unique_f_compat_tableau} and \ref{lem:unique_f_compat_tableau_tech} will concern shapes that are northeast- and northwest-hook closed. We will prove a geometric property of such shapes, which will be used later.

\begin{lemma}\label{lemma4}
Let $\lambda$ be a northeast- and northwest-hook-closed shape with at least one box in the first row. If $(i_0,j_0)\in\lambda$, then $(i,j_0)\in\lambda$ for any $1\leq i\leq i_0$.
\end{lemma}
\begin{proof}
Suppose $\lambda$ contains $(1,a)$. If $a\leq j_0$, then apply the northwest-hook-closure property of $\lambda$. Otherwise, apply the northeast-hook-closure property.
\end{proof}

\subsection{Repetitive reading words}
In this section, we define the notion of a repetitive reading word of a filling. This will be an invariant of $K$-Knuth equivalence for increasing fillings of northwest- and southeast-hook-closed shapes, so the study of such a notion will be useful.

First recall the definition of reading word given in Section \ref{sec:reading_word}. We can easily generalize this definition to any filling $F$ of any (not necessarily straight or skew) shape in the obvious way so that it coincides with the definition for increasing straight or skew tableaux. For example, define the row word of filling $F$, $\row(F)$, to be the word obtained by reading the labels of the boxes of $F$ from left to right along each row, starting from the bottom row and moving upward. We can now define repetitive reading words for a general filling $F$.

\begin{definition}
If the shape of filling $F$ is empty, then the empty word is the only repetitive reading word of $F$. In general, the word $w=tw'$ is a repetitive reading word for $F$ if and only if some southwest-most box $\beta$ of $F$ contains the letter $t$ and $w'$ is a repetitive reading word of either $F$ or $F$ with box $\beta$ removed. 
\end{definition}

\begin{example}
Consider $F=$\begin{ytableau} 1 & 2 & 4 \\ 3 & 5 & 6 \\ \none & \none & 7\end{ytableau}. Any repetitive reading word of $F$ must have length at least seven. Some examples of words of this length are 7356124 (the row word of $F$), 3152764 (the column word of $F$), and 3715624. Other repetitive reading words include 3357516264 and 3751576244.
\end{example}


The following proposition clearly follows from the definition.

\begin{proposition}\label{prop:reading_compat}
Any reading word for a filling $F$ is a repetitive reading word of $F$. In particular, the row word $\row(F)$ of a filling is a repetitive reading word.
\end{proposition}

We will prove the following proposition, which gives new invariants of $K$-Knuth equivalence relations.

\begin{proposition}\label{prop:f_compat_invar}
Let $F$ be an increasing filling of a northwest- and southeast-hook-closed shape $\lambda$.
If $u\equiv v$ and $u$ is a repetitive reading word of $F$, then $v$ is also a repetitive reading word of $F$.
\end{proposition}

Note that this proposition says that if the shape of $F$ is northwest and southeast-hook-closed, then the set of repetitive reading words of $F$ is a union of $K$-Knuth classes. 
Before proving the proposition, we need the following technical lemma.
\begin{lemma}\label{lem:f_compat_invar_tech}
Let $F$ be an increasing filling of a northwest- and southeast-hook-closed shape $\lambda$. If $\alpha\neq\beta$ are boxes in $\lambda$ and $F(\alpha) = F(\beta)$, then there are at least two letters between any $F(\alpha)$ and any $F(\beta)$ in a repetitive reading word of $F$.
\end{lemma}
\begin{proof}
Assume that $\alpha = (i_1,j_1)\neq\beta = (i_2,j_2)$. Since $F$ increases along rows and columns,
$\alpha$ is either strictly northeast or southwest of $\beta$.
In either case, using the fact that $\lambda$ is
northwest- and southeast-hook closed, the boxes $\gamma = (i_1,j_2)$ and $\delta = (i_2,j_1)$ are both in $\lambda$.
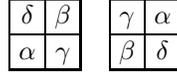
\[
\begin{ytableau}
\delta & \beta \\
\alpha & \gamma
\end{ytableau}\hspace{4mm}\begin{ytableau}
\gamma & \alpha \\
\beta & \delta
\end{ytableau}
\]
By definition, the labels in boxes $\gamma$ and $\delta$ must lie between any $F(\alpha)$ and any $F(\beta)$ in any repetitive reading word, as desired.
\end{proof}

Using this lemma, we can prove Proposition \ref{prop:f_compat_invar}.

\begin{proof}[Proof of Proposition \ref{prop:f_compat_invar}]
It suffices to assume that $v$ differs from $u$ by one $K$-Knuth move.

If $v$ is obtained from $u$ by replacing an occurrence of $p$ by $pp$, it is clear that $v$ is a repetitive reading word.

If $v$ is obtained from $u$ by replacing $pp$ by $p$, then by Lemma \ref{lem:f_compat_invar_tech}, $v$ is a repetitive reading word since the occurrence of a double $p$ in $u$ resulted from listing the same box twice.

If $v$ is obtained from $u$ by replacing $pqp$ with $qpq$, then by Lemma \ref{lem:f_compat_invar_tech}, both $q$'s label the same box of $\lambda$. Thus neither the box labeled by this $q$ nor the box labeled by this $p$ is weakly southwest of the other, so $v$ is a repetitive reading word.

If $v$ is obtained from $u$ by replacing $xzy$ by $zxy$ for $x<y<z$, let $\beta_x$, $\beta_y$, and $\beta_z$ be the boxes of $\lambda$ corresponding to these occurrences $x$, $y$, and $z$, respectively. It suffices to show that $\beta_z$ is strictly southeast of $\beta_x$. It is impossible for $\beta_z$ to be weakly southwest of $\beta_x$ since $u$ is a repetitive reading word, and it is impossible for $\beta_z$ to appear weakly northwest of $\beta_x$ because $F$ is increasing and $x<z$. Assume then, for the sake of contradiction, that
$\beta_z = (i_3,j_3)$ is weakly northeast of $\beta_x = (i_1,j_1)$.
Since the labels of the two boxes are consecutive in repetitive reading word $u$ and $\lambda$ is southeast-hook-closed,
one must have $i_3=i_1$ and $j_3=j_1+1$.
\[
\begin{ytableau}
\beta_x &\beta_z
\end{ytableau}
\]
The box~$\beta_y=(i_2,j_2)$ cannot be weakly northwest
of~$\beta_x$ or weakly southeast of~$\beta_z$
because $x<y<z$, and $\beta_y$ cannot be weakly southwest of $\beta_x$ by the definition of repetitive reading word. Since the labels of $\beta_z$ and $\beta_y$ are consecutive in $u$ and $\lambda$ is northwest-hook-closed, one must have $i_2=i_3-1$ and $j_2=j_3$.
\[
\begin{ytableau}
\none & \beta_y \\
\beta_x & \beta_z
\end{ytableau}
\]
Let $\alpha=(i_2,j_1)$,
which exists because $\lambda$ is northwest-hook closed.
\[
\begin{ytableau}
\alpha & \beta_y \\
\beta_x & \beta_z
\end{ytableau}
\]
Then $F(\alpha)$~lies between $\beta_x$ and~$\beta_y$ in $u$ by the definition of repetitive reading word, giving a contradiction.

Analogous arguments apply to the other three $K$-Knuth moves.
\end{proof}

\subsection{Proof of Theorem~\ref{theorem1}}

In this section, we will use Proposition \ref{prop:f_compat_invar} to prove Theorem~\ref{theorem1}. In particular, we will prove the following lemma.

\begin{lemma}\label{lem:unique_f_compat_tableau}
If $F$ is an increasing filling of a northeast- and northwest-hook-closed shape $\lambda$, then there is exactly one straight tableau $T$ such that $\row(T)$ is a repetitive reading word of $F$.
\end{lemma}

We first use Lemma \ref{lem:unique_f_compat_tableau} to prove Theorem~\ref{theorem1}.

\begin{proof}[Proof of Theorem~\ref{theorem1}]
Suppose $T$ is a right-alignable tableau, and recall that $T_R$ denotes the right alignment of $T$. Since $\row(T) = \row(T_R)$, $\row(T)$~is a repetitive reading word of $T_R$ by Proposition \ref{prop:reading_compat}. If $T\equiv T'$,
then $\row(T')$~is also a repetitive reading word of $T_R$ by Proposition~\ref{prop:f_compat_invar}, so $T'=T$ by Lemma~\ref{lem:unique_f_compat_tableau}. Hence, $T$~is a URT.
\end{proof}

It remains to prove Lemma \ref{lem:unique_f_compat_tableau}. Before proving it, we will need a short technical lemma. In the rest of this section, we will adopt the following notation: for (not necessarily skew or partition) shape $\lambda$, let $R_{i_0}(\lambda)$ denote the $i_0$th row of $\lambda$, i.e., $R_{i_0}(\lambda) = \left\{(i,j)\in\lambda : i = i_0\right\}$. Also, we write
$$ R_{\geq i_0}(\lambda) = \bigcup\limits_{i\geq i_0}R_i(\lambda). $$

\begin{lemma}\label{lem:unique_f_compat_tableau_tech}
Let $F$ be an increasing filling of a northeast- and northwest-hook-closed shape~$\lambda$ with at least one box in the first row. Suppose $w=w_1w_2\cdots w_n$ is a repetitive reading word of $F$, and let $w_j^\lambda$ denote the box in $\lambda$ that contributes the letter $w_j$ to $w$. For any $j$ and $k$, the following are equivalent.
\begin{enumerate}
	\item\label{lem:unique_f_compat_tableau_tech1} $w_j$~is the first letter of a strictly decreasing subword of $w$ of length~$k$.
	\item\label{lem:unique_f_compat_tableau_tech2} $w_j^\lambda\in R_{\geq k}(\lambda)$.
\end{enumerate}
\end{lemma}
\begin{proof}
Assume (\ref{lem:unique_f_compat_tableau_tech1}). Let $j = a_1 < a_2 <\cdots < a_k$ be such that $w_{a_1} > w_{a_2} > \cdots > w_{a_k}$. The box~$w_{a_{r+1}}^\lambda$ cannot be weakly southeast of the box~$w_{a_r}^\lambda$ because $w_{a_{r+1}} < w_{a_r}$, and $w_{a_{r+1}}^\lambda$ cannot be weakly southwest of $w_{a_r}^\lambda$ by definition of repetitive reading word. Hence $w_{a_{r+1}}^\lambda$ is strictly north of $w_{a_r}^\lambda$.  Thus (\ref{lem:unique_f_compat_tableau_tech2}) follows.

Conversely, assume  (\ref{lem:unique_f_compat_tableau_tech2}). Write $w_j^\lambda = (p,q)$ with $p\geq k$. By Lemma \ref{lemma4}, the box $(p-r+1,q)$ is in $\lambda$ for $1\leq r\leq k$, so since $w$ is a repetitive reading word, there exist $j = a_1 < a_2 < \cdots < a_k$ such that $w_{a_r}^\lambda = (p-r+1,q)$. Since $F$ increases down columns, $w_{a_1} > w_{a_2} > \cdots > w_{a_k}$, so (\ref{lem:unique_f_compat_tableau_tech1}) holds.
\end{proof}

Using Lemma \ref{lem:unique_f_compat_tableau_tech}, we can then complete the proof of Lemma \ref{lem:unique_f_compat_tableau}. Recall that Lemma \ref{lem:unique_f_compat_tableau} claims that there is exactly one straight tableau $T$ such that $\row(T)$ is a repetitive reading word for $F$.

\begin{proof}[Proof of Lemma \ref{lem:unique_f_compat_tableau}]
By moving the shape $\lambda$ if necessary, we may assume that $\lambda$ has at least one box in the first row.

We first show that there is at most one such straight tableau. Suppose $T$ is a straight shape tableau of shape $\mu$ and $w=\row(T)=w_1w_2\cdots w_n$ is a repetitive reading word for $F$. We will show that $T$ is uniquely determined by this condition. Recall from Proposition \ref{prop:reading_compat} that $w$ is a repetitive reading word for $T$, and again let $w_j^\mu$ (resp. $w_j^\lambda$) denote the box of $\mu$ (resp. $\lambda$) that contributes $w_j$ to $w$.  

Note that the straight shape $\mu$ is northeast- and northwest-hook-closed and has at least one box in the first row, so by applying Lemma \ref{lem:unique_f_compat_tableau_tech} twice, we see that $w_j^\lambda\in R_{\geq i}(\lambda)$ if and only if $w_j$~is the first letter of a strictly decreasing subword of $w$ of length~$i$, which happens if and only if $w_j^\mu\in R_{\geq i}(\mu)$. Since $R_i(\lambda) = R_{\geq i}(\lambda) - R_{\geq(i+1)}(\lambda)$ (and similarly for $\mu$), $w_j^\lambda\in R_i(\lambda)$ if and only if $w_j^\mu\in R_i(\mu)$.

Since $F(w^\lambda_1)F(w^\lambda_2)\cdots F(w^\lambda_n) = w = T(w^\mu_1)T(w^\mu_2)\cdots T(w^\mu_n)$, this means that the letters in the $i$th row of $T$ are exactly the letters in the $i$th row of $F$. Since $T$ and $F$ are both increasing along rows, $T$ is uniquely determined.

We now describe the straight tableau $T$ with $\row(T)$ a repetitive reading word for $F$. Let $F_L$ denote the left alignment of increasing filling $F$ defined in the natural way (left-justify each row). Since $F$ is northeast- and northwest-hook-closed, the sizes of the rows of $F$ must be weakly decreasing from top to bottom, and hence the same is true for $F_L$.

To see that $F_L$ is increasing, suppose we have boxes $\alpha$ and $\beta$ of $F$ such that $\alpha$ is directly above $\beta$ in $F_L$. Then $\alpha$ must be weakly northwest of $\beta$ in $F$. It follows that the northeast hook between $\alpha$ and $\beta$ exists in $F$, which forces $F(\alpha)<F(\beta)$. Since it is clear that $F_L$ is increasing along rows, we conclude that $F_L$ is increasing.

Therefore, taking $T=F_L$ gives a straight tableau with $\row(T)=\row(F)$, a repetitive reading word of $F$.
\end{proof}

\subsection{Shapes of Non-URTs}
\ytableausetup{mathmode, boxsize=2em, centertableaux}
Recall that Corollary \ref{cor:rect_URT} says that every tableau of rectangular shape is a URT. One might wonder if there are other shapes that are always URTs regardless of filling. It is easily checked that every tableau of shape $\tiny\yng(2,1)$ is a URT. We will show in this section that these are all the possibilities.

\begin{proposition} \label{prop:sharpRect}
If $\lambda$ is a straight shape
that is not a rectangle or $\tiny\yng(2,1)$
then there is an increasing tableau
of shape $\lambda$ that is not a URT.
\end{proposition}

We will divide the proof into a few steps. Our first step is to construct pairs of $K$-Knuth equivalent two-row tableaux, which will serve as the building block in subsequent steps.

For any $k > 0$ and $j\geq k+2$, define the tableau
\[
T_{j,k} = 
\begin{ytableau}
\sss 1 & \cdots & \sss k-1 & *(lightgray) \sss k & *(lightgray) \sss k+1 & *(lightgray) \sss k+3 
& \sss k+4 & \cdots & \sss j+1 \\
\sss 2 & \cdots & \sss k & *(lightgray) \sss k+2
\end{ytableau},
\]
of shape $(j,k)$, as well as the tableau
\[
T_{j,k}' = 
\begin{ytableau}
\sss 1 & \cdots & \sss k-1 & *(lightgray) \sss k & *(lightgray) \sss k+1 & *(lightgray) \sss k+3 
& \sss k+4 & \cdots & \sss j+1 \\
\sss 2 & \cdots & \sss k & *(lightgray) \sss k+2 & *(lightgray) \sss k+3
\end{ytableau},
\] of shape $(j,k+1)$
(where there are no boxes to the left of the gray boxes in the case $k=1$). We will prove the following lemma.
\begin{lemma}\label{lem:sharpRect1}
$T_{j,k}\equiv T_{j,k}'$.
\end{lemma}
\begin{proof}
Note that $T_{j,k}$ can be obtained by inserting $(k+1)$ once into the tableau
\[
T = \begin{ytableau}
\sss 1 & \cdots & \sss k-1 & *(lightgray) \sss k &
*(lightgray)\sss k+2 & *(lightgray) \sss k+3 
& \sss k+4 & \cdots & \sss j+1 \\
\sss 2 & \cdots & \sss k
\end{ytableau}\qquad
\]
(where there is no box in the second row in the case $k=1$), while $T_{j,k}'$ can be obtained by inserting $(k+1)$ twice into $T$.
\end{proof}

Our next step is to prove a lemma that allows us to reduce to the case of two-line tableaux. Given a tableau~$T$ of shape $\lambda=(\lambda_1,\dots,\lambda_\ell)$,
let~$T^{(r,s)}$ denote the restriction of~$T$
to the rows $r,r+1,\dots,s$ of~$\lambda$. Similarly, we let $\lambda^{(r,s)}$ denote the restriction of~$\lambda$
to these rows.

\begin{lemma}\label{lem:sharpRect2}
If $T$ and~$T'$ are two tableaux that differ
only in the $i$th~and $(i+1)$th rows,
and if $T^{(i,i+1)}\equiv (T')^{(i,i+1)}$,
then $T\equiv T'$.
\end{lemma}
\begin{proof}
Note that $\row(T)$ and~$\row(T')$
may be connected by $K$-Knuth moves
that only use letters from the $i$th~and
$(i+1)$th rows.
\end{proof}

We can now complete the proof.

\begin{proof}[Proof of Proposition \ref{prop:sharpRect}]
For any tableau~$T$, we will denote by $T[n]$
the tableau formed by increasing the entries of~$T$ by $n$:
formally, $T[n](i,j) = T(i,j) + n$.

We first consider the following special case: there is an index~$i$ such that
$\lambda_i \geq \lambda_{i+1} + 2$ and $\lambda_{i+1} > 0$.
It is easy to see that there exists a tableau~$T$ of shape $\lambda$
such that $T^{(i,i+1)} = T_{\lambda_i,\lambda_{i+1}}[n]$
for some~$n$. Let~$T'$ denote the tableau of shape 
$(\dots,\lambda_i,\lambda_{i+1}+1,\lambda_{i+2},\dots)$
that has the same labels as~$T$ in 
every row except the $i$th and $(i+1)$th and
such that $(T')^{(i,i+1)} = T'_{\lambda_i,\lambda_{i+1}}[n]$.
By Lemma \ref{lem:sharpRect1}, $T^{(i,i+1)}\equiv (T')^{(i,i+1)}$, so by Lemma \ref{lem:sharpRect2}, $T\equiv T'$. Clearly $T\neq T'$, as desired.

By an analogous construction
using the transposes $T_{j,k}^t$
and~$(T'_{j,k})^t$, 
we see that the proposition holds for any shape $\lambda$ such that the following is true for either $\mu=\lambda$ or $\mu=\lambda^t$: there exists $i$ such that $\mu_i \geq \mu_{i+1}+2$ and $\mu_{i+1}>0$.

One can easily check that the only cases not covered by the above argument are the shapes $\lambda = (k^i,k-1,k-2,\ldots,k-m)$ with $k > m\geq 1$ (where $k^i$ denotes $i$ rows of length $k$). If $k\geq 3$, then $\lambda^{(i,i+1)} = (k,k-1)$ so we may use the same argument as above with a tableau $T'$ of shape $\lambda$ for which $(T')^{(i,i+1)} = T'_{k,k-2}[n]$ for some $n$. If $i+m\geq 3$, then $\lambda^t = ((i+m)^{k-m},i+m-1,\ldots,i)$ so we reduce to the previous case. If $i+m\leq 2$ and $k=2$, $\lambda = \tiny\yng(2,1)$, which we do not consider.
\end{proof}

\section{Hook-Shaped Tableaux}
In this section, we examine a class of tableaux
known as the hook-shaped tableaux and
characterize which hook-shaped tableaux are URTs.

\begin{definition} \label{def:hook}
A straight shape $\lambda$ is \emph{hook-shaped}
if $\lambda = (m,1^n)$ for some $m\geq1$ and $n\geq0$.
An increasing tableau $T$ of shape $\lambda$
is \emph{hook-shaped} if $\lambda$ is hook-shaped.
\end{definition}

Of the tableaux below,
the tableau on the left is hook-shaped and the two tableaux on the right are not.
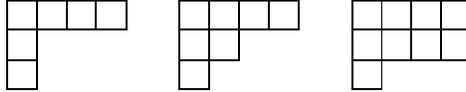
\[
\yng(4,1,1)\qquad
\yng(4,2,1)\qquad
\yng(4,4,1)
\]

\begin{definition} \label{def:arm,leg}
Let $T$ be a hook-shaped tableau.
We define $\arm(T)$ to be the ordered tuple of labels in the first row of~$T$, excluding the leftmost box. We define $\leg(T)$ to be the ordered tuple of lables in the first column of~$T$, excluding its northernmost box.
\end{definition}

\begin{example}
The following tableau has $\arm(T)=(2,3,4)$ and $\leg(T)=(5,6).$
\[
T=\young(1234,5,6)\qquad
\]
\end{example}

\begin{theorem} \label{claim:1}
Let $T$ be an initial, hook-shaped
tableau. Then $T$ is a URT if and only if both $\arm(T)$ and $\leg(T)$ have
consecutive entries.

\end{theorem}
\begin{proof}
Assume both $\arm(T)$ and $\leg(T)$ have consecutive entries. If suffices to consider initial, hook-shaped tableaux with first row $1,2,\ldots,k$ for some $k$ because an initial, hook-shaped tableau with consecutive arm and leg has first row and/or first column $1,2,\ldots,k$ and classes are preserved by taking transposes.

Let $T$ be such a tableau with $\arm(T)=(2,3,\ldots k)$ and $\leg(T) = (\ell, \ell+1,\ldots, \ell+t)$. If $\ell=2$, $T$ is minimal and hence a URT, so we assume $\ell\geq3$. We proceed by induction on the length of the first row.

If the first row has only one box, $T$ is rectangular, which implies it is a URT by Corollary~\ref{cor:rect_URT}.

Assume every suitable tableau with first row of length $n$ is a URT, and consider $T$ with first row $1,2,\ldots, n+1$. Suppose $T'$ is a tableau with $T'\equiv T$. Then $T$ and $T'$ have the same arm and leg by Proposition \ref{prop:outer_hook}, and $T\!\mid_{[2,n+1]}\equiv T'\!\mid_{[2,n+1]}$ by Proposition \ref{restrictedalpha}.

Perform a forward $K$-jdt slide on $T\!\mid_{[2,n+1]}$
and $T'\!\mid_{[2,n+1]}$ at position
$(1,1)$, that is, at the position vacated by $1$
upon restriction to the subalphabet $[2,n+1]$.
Because the first rows of $T\!\mid_{[2,n+1]}$
and $T'\!\mid_{[2,n+1]}$ are labeled consecutively,
$\ell\geq3$, and both tableaux are increasing with the same outer hook,
performing this $K$-jdt slide translates the first row of each tableau one box to the left.
Let $S$ and $S'$ denote this $K$-rectification of $T\!\mid_{[2,n+1]}$
and $T'\!\mid_{[2,n+1]}$, respectively. Note that $S\equiv S'$.

Let $S[-1]$ and $S'[-1]$ denote the tableaux obtained by subtracting 1 from each entry of $S$ and $S'$, respectively. Then $S\equiv S'$ implies $S[-1]\equiv S'[-1]$. And $S[-1]$ is a URT by inductive hypothesis, so $S[-1]=S'[-1]$. It follows that $T'=T$, and so $T$ is a URT.

Next, assume $T$ has at least one of $\arm(T)$ or $\leg(T)$ without consecutive entries, and we show $T$ is not a URT. Without loss of generality, we assume that every tableau $T$ in consideration has $\leg(T)$ non-consecutive since transposing preserves equivalence classes. We consider two cases. 

\begin{case} Assume that $\leg(T)$ contains 2.
\end{case}

Consider the tableaux
\ytableausetup{mathmode, boxsize=1.6em}
\begin{gather*}
\begin{ytableau} 
1 & k \\ 
2\\ 
\vdots\\
\scriptscriptstyle k-1\\
\scriptscriptstyle k+p
\end{ytableau}
\quad\equiv\quad
\begin{ytableau} 
1 & k \\ 
2& \scriptscriptstyle k+p\\ 
\vdots\\
\scriptscriptstyle k-1\\
\scriptscriptstyle k+p
\end{ytableau}
\end{gather*}
where $k \ge 3$ and $p \ge 1$. The equivalence can be seen easily and directly by performing $K$-Knuth moves on the row words of the two tableaux. For example, if $k=4$ and $p=2$, the tableau on the left has row word 63214. We see $$63214\equiv 63241 \equiv 63421 \equiv 636421 \equiv 363421 \equiv 363241$$ and the last word inserts into the tableau on the right. This pattern will work for any $k$ and $p$.

The simplified example above generalizes to all initial tableaux with the properties specified in Case 1 via row and column insertion as follows.

If we row-insert some sequence of letters greater than $k$ to the first row of each tableau, the equivalence of the two tableaux does not change; this is the same as adding a sequence of letters to the end of each of their respective row words. Hence, we have 
\ytableausetup{mathmode, boxsize=1.6em}
\begin{gather*}
\begin{ytableau} 
1 & k & \scriptscriptstyle a_{i+1} & \cdots & a_{s}\\ 
2\\ 
\vdots\\
\scriptscriptstyle k-1\\
\scriptscriptstyle k+p
\end{ytableau}
\quad\equiv\quad
\begin{ytableau} 
1 & k & \scriptscriptstyle a_{i+1} & \cdots & a_{s}\\ 
2 & \scriptscriptstyle k+p\\ 
\vdots\\
\scriptscriptstyle k-1\\
\scriptscriptstyle k+p
\end{ytableau}
\end{gather*}
Similarly, we can column-insert integers less than $k-1$ into the first column of each tableau and maintain equivalence. Thus we can obtain any sequence of integers between $1$ and $k$ in the first row. For example, suppose we want to obtain $a_{1}, ... , a_{i-1}$ between 1 and $k$. We can first column-insert $a_{i-1}-1$ into the first column. This has the effect of shifting everything to the right of $1$ in the first row one box to the right, and inserting $a_{i-1}$ into the position $(1, 2)$. We can repeat this process, inserting $a_{i-2}-1$ into the first column, and so on, until we have the following equivalence:
\ytableausetup{mathmode, boxsize=1.6em}
\begin{gather*}
\begin{ytableau} 
1 & a_1 & \cdots & \scriptscriptstyle \scriptscriptstyle a_{i-1}& k & a_{i} & \cdots & a_{s}\\ 
2\\ 
\vdots\\
\scriptscriptstyle k-1\\
\scriptscriptstyle k+p
\end{ytableau}
\quad\hspace{-.8cm}\equiv\quad
\begin{ytableau} 
1 & a_1 & \cdots & \scriptscriptstyle a_{i-1} & k & a_{i} & \cdots & a_{s}\\ 
2 & \scriptscriptstyle k+p\\ 
\vdots\\
\scriptscriptstyle k-1\\
\scriptscriptstyle k+p
\end{ytableau}.
\end{gather*}

Finally, we can column-insert any sequence of integers larger than $k+1$ in the first column. Hence we have the equivalence

\ytableausetup{mathmode, boxsize=1.6em}
\begin{gather*}
\begin{ytableau} 
1 & a_1 & \cdots & \scriptscriptstyle a_{i-1}& k & a_i & \cdots & a_{s}\\ 
2\\ 
\vdots\\
\scriptscriptstyle k-1\\
\scriptscriptstyle k+p\\
\ell_{j}\\
 \vdots\\
\ell_{t}
\end{ytableau}
\quad\hspace{-.8cm}\equiv\quad
\begin{ytableau} 
1 & a_1 & \cdots & \scriptscriptstyle a_{i-1}& k & a_i & \cdots & a_{s}\\ 
2 & \scriptscriptstyle k+p\\ 
\vdots\\
\scriptscriptstyle k-1\\
\scriptscriptstyle k+p\\
\ell_{j}\\
\vdots\\
\ell_{t}
\end{ytableau}.
\end{gather*}


\begin{case} Assume $\leg(T)$ does not contain 2. 
\end{case}

Since $T$ is initial by assumption, $\arm(T)$ contains $2$. If $\arm(T)$ is non-consecutive we may apply Case 1 to $T^t$. Thus we may assume $\arm(T)=(2,3,\ldots,k)$.

Consider the tableau  

\ytableausetup{mathmode, boxsize=1.6em}
\begin{gather*}
\begin{ytableau} 
1 & 2 & 3 & \cdots & k \\ 
\scriptscriptstyle k-1\\
\scriptscriptstyle k+p
\end{ytableau}.
\end{gather*}

We show by induction that this is not a URT for all $k \ge 3$. When $k = 3$ the tableau is not a URT due to the below equivalence, which can be easily checked.
\ytableausetup{mathmode, boxsize=1.6em}
\begin{gather*}
\begin{ytableau} 
1 & 2 & 3 \\ 
2\\ 
4\
\end{ytableau}
\quad\equiv\quad
\begin{ytableau} 
1 & 2 & 3 \\ 
2 & 4\\ 
4\
\end{ytableau}.
\end{gather*}

Assume that for some $k$, we have the equivalence: 
\ytableausetup{mathmode, boxsize=1.6em}
\begin{gather*}
\begin{ytableau} 
1 & 2 & 3 & \cdots & k\\ 
\scriptscriptstyle k-1\\
\scriptscriptstyle k+p
\end{ytableau}
\quad\equiv\quad
\begin{ytableau} 
1 & 2 & 3 & \cdots & k\\ 
\scriptscriptstyle k-1 & \scriptscriptstyle k+p\\
\scriptscriptstyle k+p
\end{ytableau}.
\end{gather*}

We show that
\ytableausetup{mathmode, boxsize=1.6em}
\begin{gather*}
\begin{ytableau} 
1 & 2 & 3 &\cdots & \scriptscriptstyle k+1\\ 
k\\
{\parbox{1.4em}{\tiny $k+$ \\ $p+1$}}\\
\end{ytableau}
\quad\equiv\quad
\begin{ytableau} 
1 & 2 & 3 &\cdots & \scriptscriptstyle k+1\\ 
k & {\parbox{1.4em}{\tiny $k+$ \\ $p+1$}}\\
{\parbox{1.4em}{\tiny $k+$ \\ $p+1$}}\\
\end{ytableau}
\end{gather*}

by examining $K$-Knuth equivalence moves on their row words.\\

First, consider the row word of the tableau on the left-hand side of the above equivalence. We have, with commas separating distinct letters for ease of reading:

\[k + (p + 1), k, 1, 2, \ldots , k + 1 \equiv 1, k + (p + 1), k, 2, \ldots , k + 1.\]

In the word above, consider everything to the right of $1$. Standardize it to obtain the word
\[k+1,k-1,1,\ldots,k.\]
By assumption, this word is $K$-Knuth equivalent to 
\[k+1,k-1,k+1,1,\ldots,k.\]
Thus 
\[1,k+(p+1),k,2,\ldots,k+1  \equiv 1,k+(p+1),k,k+(p+1),2,\ldots,k+1.\]
Since these two words respectively insert into the two tableaux depicted above, the tableaux are equivalent.

To generalize this equivalence to all tableaux described in Case 2, use row and column insertion as in the proof of Case 1. 
\end{proof}

\begin{example}
These hook-shaped tableaux have consecutive arm and leg and are therefore URTs:
\[
\young(1234,3,4,5)\qquad
\young(12345,3)\qquad
\young(12345,5,6).
\]
These hook-shaped tableaux do not have both consecutive arm and leg and are therefore not URTs:
\[
\young(1235,4)\qquad
\young(1234,2,4)\qquad
\young(1234,2,3,5).
\]
\end{example}

We have the following result from the proof of Theorem~\ref{claim:1}.

\begin{theorem} \label{claim:4}
If the hook-shaped tableaux
\begin{gather*}
\ytableausetup{boxsize=\boxsize}
T = \begin{ytableau} 
1 & a_{2} & \cdots & a_{n}\\
b_{2}\\
\vdots\\
b_{n}\\
\end{ytableau}
\quad\text{and}\quad T'=
\begin{ytableau} 
1 & a_{2} & \cdots & a_{n}\\
b_{2} & a_{i}\\
\vdots\\
b_{n}\\
\end{ytableau}
\end{gather*}
are initial and $a_i - a_{i-1} > 1$ then $T\equiv T'$.
\end{theorem}

\section{Conjectures and Related Results}
\subsection{Sizes of Tableaux Classes}
In the course of studying the $K$-Knuth
equivalence relation on tableaux,
we computed all equivalence classes
of tableaux on~$[n]$ for $0\leq n\leq 7$.
We were unable to obtain asymptotic bounds on the size of $K$-Knuth equivalence classes, but they seem to grow 
at least as quickly as $n!$.

\begin{table}[h] 
\centering
\caption{Sets of Initial Tableaux\label{table:1}}
\begin{tabular}{cccc}
\toprule
Alphabet Size
& \parbox[c]{0.9in}{Initial Increasing Tableaux}
& \parbox[c]{1.1in}{$K$-Knuth Classes of Initial Tableaux}
& URTs \\
\midrule
0 & 1 & 1 & 1\\
1 & 1 & 1 & 1\\
2 & 3 & 3 & 3\\
3 & 13 & 13 & 13\\
4 & 87 & 79 & 71\\
5 & 849 & 620 & 459\\
6 & 11915 & 6036 & 3313\\
7 & 238405 & 70963 & 25904\\
\bottomrule
\end{tabular}
\end{table}

Table~\ref{table:1} shows that the ratio
of unique rectification classes of tableaux on~$[n]$
to all $K$-Knuth classes of tableaux on~$[n]$
decreases monotonically,
and we expect the ratio to asymptotically
tend to zero.

\begin{conjecture}
Let $I_n$ denote the number of $K$-Knuth
equivalence classes of initial tableaux on~$[n]$,
and let $U_n$ denote the number of URTs
on~$[n]$. Then
$\lim_n U_n/I_n = 0$.
\end{conjecture}

\subsection{Composition of $K$-Knuth Classes of Tableaux}
\begin{proposition}
\label{tableauxon2n}
For every $n\geq2$, there is an equivalence class
of tableaux on~$[2n]$
containing at least $n!$ distinct tableaux.
\end{proposition}
\begin{proof}
\ytableausetup{mathmode, boxsize=2em,centertableaux}
By Theorem~\ref{claim:4}, for every $k = 2,3,\dots,n$ there is an equivalence
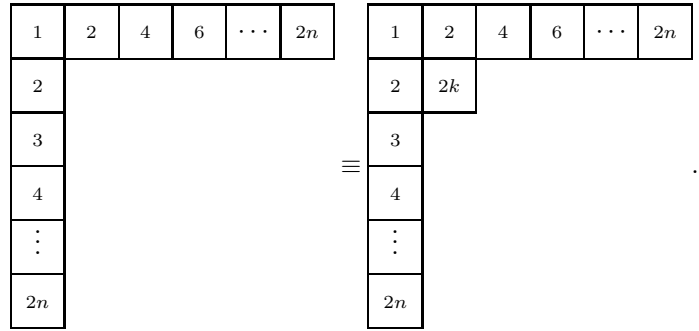
\[
\begin{ytableau}
\sss 1 & \sss 2 & \sss 4 & \sss 6 & \cdots & \sss 2n \\
\sss 2 \\
\sss 3 \\
\sss 4\\
\vdots \\
\sss 2n
\end{ytableau} \equiv 
\begin{ytableau}
\sss 1 & \sss 2 & \sss 4 & \sss 6 & \cdots & \sss 2n \\
\sss 2 & \sss 2k \\
\sss 3 \\
\sss 4 \\
\vdots \\
\sss 2n
\end{ytableau}.
\]
Let $T$ denote the tableau on the left,
let $T'$ denote the tableau on the right,
and let $w=\row(T)$.
A simple computation shows that $P(\row(T))=T$, 
and row inserting $2k-2$ into $T$
shows that $T'=P(w(2k-2))$. 
Hence $w=\row(T)\equiv\row(T')\equiv w(2k-2)$.

It follows that if we Hecke insert any of the positive integers
$2,4,6,\cdots,2(n-1)$ in any order into the first 
row~of $T$, the resulting tableau lies in the same equivalence class
as~$T$.
Hence any tableau~$T'$ with the following
two properties is $K$-Knuth equivalent to~$T$.
\begin{enumerate}
\item
The first row and the first column of~$T'$
agree with the first row and the first column of~$T$, respectively.
\item
Let $U$ be the tableau obtained by
removing the first row and first column from $T'$.
Then the tableau~$U$ uses the letters
$4,6,\ldots,2n.$
\end{enumerate}
There are at least $n!$ possibilities for
the tableau~$U$, as we saw in Section~\ref{section:invariants}.
Thus the class of tableaux $K$-Knuth equivalent
to $T$ contains at least $n!$ tableaux.
\end{proof}

The process for generating tableaux 
described in the proof of Proposition~\ref{tableauxon2n} produces many tableaux 
in the equivalence class of a hook-shaped tableau.
It suggests an important relationship between
$K$-Knuth equivalence and row insertion,
a relationship we have yet to fully understand.

\begin{proposition} \label{prop:hookClass}
Let $T$ be an initial, hook-shaped tableau
with $\arm(T)=(a_1,a_2,\ldots,a_k)$ and $\leg(T)=(b_1,b_2,\ldots,b_t)$. 

Let
\begin{align*}
A&=\{a_i\,:\,i\geq1,\,a_{i+1}-a_i\geq 2\}\\
B&=\{b_i\,:\,i\geq1,\,b_{i+1}-b_i\geq 2\}.
\end{align*}
Then the set of straight tableaux that are $K$-Knuth
equivalent to~$T$ includes all the tableaux  obtained by making the following
insertions into~$T$, in any order:
(1)~row inserting elements of~$A$ into the first row of~$T$
and (2)~column inserting elements of~$B$ into
the first column of~$T$.
\end{proposition}
\begin{example}
The following tableaux are $K$-Knuth equivalent.
They may all be obtained by row inserting $2$
and column inserting~$3$.
\[
\young(1245,2,3,5,6)\qquad\young(1245,24,3,5,6)\qquad
\young(1245,25,3,5,6)\qquad\young(1245,24,35,5,6)
\]
Proposition~\ref{prop:hookClass} does not give all
tableaux in a class. The tableaux above
are also equivalent to the tableau
\[
\young(1245,245,3,5,6),
\]
which cannot be obtained by making the described row and column insertions.
\end{example}

Proposition ~\ref{tableauxon2n} shows that the maximum size of a $K$-Knuth equivalence class
of tableaux on~$[n]$ is unbounded as $n$ increases.
In fact, the maximum number of standard tableaux in a $K$-Knuth equivalence class is also unbounded.

\begin{proposition}
For every $n > 0$, there exists a $K$-Knuth class containing at least $2^n$ standard tableaux.
\end{proposition}
\begin{proof}
We sketch the construction and leave
the details to the reader.
Consider first the tableaux
\[
T = \young(125,347,6)\quad\text{ and }\quad T' = \young(125,34,67).
\]
Note that $T\equiv T'$ since
$36731452\equiv 36734152$ and the
two words insert into $T$ and $T'$ respectively.

Using $3\times 3$ blocks, we build a standard tableau $U$ that has $3n$ rows, $3n$ columns, and the rough shape of a triangle. The outermost $3\times3$-block 
diagonal of $U$ will
consist of $n$ $7$-box tableaux
that have standardization $T$.
For instance, if $n = 3$, one possible construction of $U$ is
\[
\ytableausetup{boxsize=\boxsize}
\newcommand{\gr}{*(lightgray)}
\begin{ytableau}
\gr1  &\gr2  &\gr3  &10 &11 &12 &\gr19 &\gr20 &\gr23 \\
\gr4  &\gr5  &\gr6  &13 &14 &15 &\gr21 &\gr22 &\gr25 \\
\gr7  &\gr8  &\gr9  &16 &17 &18 &\gr24 \\
26 &27 &28 &\gr35 &\gr36 &\gr39 \\
29 &30 &31 &\gr37 &\gr38 &\gr41 \\
32 &33 &34 &\gr40 \\
\gr42 &\gr43 &\gr46 \\
\gr44 &\gr45 &\gr48 \\
\gr47
\end{ytableau}.
\]
Then $U$ is equivalent
to any tableau formed from $U$
by replacing one of the $7$-box
tableaux with an equivalent tableau
that has standardization~$T'$.
There are $2^n$ tableaux
that can result from $U$
by making a sequence of
these modifications, and
by construction each tableau is
standard and equivalent
to the others.
\end{proof}

\subsection{Shapes of Tableaux}
Which shapes appear in a $K$-Knuth class of tableaux?
We initially suspected that each tableau
class contains a minimum and maximum shape,
ordering the shapes under inclusion,
but the following class disproves our conjecture:
\[
\young(125,236,3,4,5)\qquad\young(125,23,36,4,5)\qquad\young(125,236,36,4,5)\qquad\young(125,23,35,46,5).
\]
However, it seems -- and we have not been able to find
a counterexample -- that if two shapes
$\lambda_1\subseteq\lambda_2$ appear among
the shapes in an equivalence class,
then every shape in the interval~$[\lambda_1,\lambda_2]$
of Young's lattice appears among the shapes in that class.

\begin{conjecture}
Let $T_1,T_2,\dots,T_k$ be a
$K$-Knuth class of straight tableaux
and let $T_i$~have shape~$\lambda_i$.
Then the set $\Sigma=\{\lambda_1,\lambda_2,\dots,\lambda_k\}$
has the following property:
if $\lambda_i,\lambda_j\in\Sigma$
then $[\lambda_i,\lambda_j]\subseteq\Sigma$.
\end{conjecture}

This conjecture has been verified for $K$-Knuth classes on~$[n]$ for $n\leq 7$.

\subsection{Changes in Tableau Shape}
Initially, we conjectured that if two words $w$ and $w'$ differ by just one $K$-Knuth move, then the shapes of $P(w)$ and $P(w')$ differ by just one box. However, we found a counterexample:
$5451342154 \equiv 54513422154$ because $2 \equiv 22$, 
but the two words insert into the following tableaux,
respectively:
\begin{gather*}
\ytableausetup{boxsize=\boxsize}
\begin{ytableau} 
1 & 2 & 4 & 5\\
2 & 5\\
3\\
4\\
5
\end{ytableau}
\quad\equiv\quad
\begin{ytableau} 
1 & 2 & 4 & 5\\
2 & 4 & 5\\
3 & 5\\
4\\
5
\end{ytableau}.
\end{gather*}

\section*{Acknowledgments}
This research was conducted at the 2014 REU program at the University of Minnesota, Twin Cities, and was supported by NSF grant DMS-1148634. For their mentorship and support, we would like to thank Joel Lewis, Gregg Musiker, Pasha Pylyavskyy, and Vic Reiner. We are especially grateful to the reviewer who gave us many valuable improvements on earlier versions of this paper.

\end{document}